\documentclass[11pt]{article}

\usepackage[english]{babel}
\usepackage{amsmath,amsfonts,amssymb,amsthm}
\usepackage{latexsym}
\usepackage{makeidx}
\usepackage{amscd}

\usepackage{vmargin}
\setmarginsrb{20mm}{10mm}{20mm}{20mm}%
            {8mm}{8mm}{8mm}{8mm}

\newtheorem{em-deff}{Definition}[section]
\newtheorem{lemma}[em-deff]{Lemma}
\newtheorem{theorem}[em-deff]{Theorem}
\newtheorem{corollary}[em-deff]{Corollary}

\newtheorem{proposition}[em-deff]{Proposition}
\newtheorem{fact}[em-deff]{Fact}
\newtheorem{em-example}[em-deff]{Example}
\newtheorem{claim}[em-deff]{Claim}
\newtheorem{problem}[em-deff]{Problem}

\newtheorem{em-remark}[em-deff]{Remark}
\newtheorem{question}[em-deff]{Question}

\newenvironment{example}{\begin{em-example} \em }{ \end{em-example}}
\newenvironment{remark}{\begin{em-remark} \em }{ \end{em-remark}}
\newenvironment{deff}{\begin{em-deff} \em }{ \end{em-deff}}

\newcommand{\VG}{\mathcal V_{(G,\tau)}(e_G)}
\newcommand{\VH}{\mathcal V_{(H,\sigma)}(e_H)}
\newcommand{\amu}{^{-1}}

\input xy
\xyoption{all}

\title{Semitopological homomorphisms}
\author{Anna Giordano Bruno \\ {\footnotesize Dipartimento di Matematica e Informatica,
Universit\`{a} di Udine} \\ {\footnotesize Via delle Scienze, 206 - 33100 Udine, Italy} \\ {\footnotesize {\tt anna.giordanobruno@dimi.uniud.it}}}
\date{}

\begin{document}

\maketitle

\begin{abstract}
Inspired by an analogous result of Arnautov about isomorphisms, we prove that all continuous surjective homomorphisms of topological groups $f:G\to H$ can be obtained as restrictions of open continuous surjective homomorphisms $\widetilde f:\widetilde G\to H$, where $G$ is a topological subgroup of $\widetilde G$. In case the topologies on $G$ and $H$ are Hausdorff and $H$ is complete, we characterize continuous surjective homomorphisms $f:G\to H$ when $G$ has to be a dense normal subgroup of $\widetilde G$. 

We consider the general case when $G$ is requested to be a normal subgroup of $\widetilde G$, that is when $f$ is \emph{semitopological} --- Arnautov gave a characterization of semitopological isomorphisms internal to the groups $G$ and $H$. In the case of homomorphisms we define new properties and consider particular cases in order to give similar internal conditions which are sufficient or necessary for $f$ to be semitopological. Finally we establish a lot of stability properties of the class of all semitopological homomorphisms.
\end{abstract}


\section{Introduction}

In \cite[Theorem 1]{Ar} Arnautov showed that for every continuous isomorphism $f:(G,\tau)\to(H,\sigma)$ of topological  groups, there exist a topological group $(\widetilde G,\widetilde\tau)$ containing $G$ as a topological subgroup and an open continuous homomorphism $\widetilde f:(\widetilde G,\widetilde\tau)\to (H,\sigma)$ extending $f$. Moreover he noted that such a pair $(\widetilde G,\widetilde f)$ needs not always exist under the additional assumption that $G$ is a normal subgroup of $\widetilde G$.
So in \cite[Definition 2]{Ar} the author defined a continuous isomorphism $f:(G,\tau)\to (H,\sigma)$ of topological groups to be \emph{semitopological} if there exist a topological group $(\widetilde G,\widetilde\tau)$ containing $G$ as a topological normal subgroup and an open continuous homomorphism $\widetilde f:(\widetilde G,\widetilde\tau)\to (H,\sigma)$ extending $f$ (the counterpart of this definition for topological rings was given in \cite{Ar0}).

If $f:G\to H$ is a group homomorphism, denote by $\Gamma_f$ the \emph{graph} of $f$, that is the subgroup $\Gamma_f=\{(g,f(g)):g\in G\}$ of $G\times H$.
Using explicitly the graph of a continuous surjective group homomorphism $f:G\to H$ we can prove \cite[Theorem 1]{Ar} in a much simpler way and even generalize it weakening the hypothesis on $f$ (in our theorem $f$ is a homomorphism while in \cite[Theorem 1]{Ar} $f$ was an isomorphism) and achieving $G$ to be a \emph{closed} subgroup of $\widetilde G$, in case the topology on the codomain is Hausdorff. 

\begin{theorem}\label{ext}
Let $G$, $H$ be topological groups and $f:(G,\tau)\to (H,\sigma)$ a continuous surjective homomorphism. Then there exist a topological group $(\widetilde G,\widetilde\tau)$ containing $(G,\tau)$ as a topological subgroup and an open continuous homomorphism $\widetilde f:(\widetilde G,\widetilde\tau)\to (H,\sigma)$ such that $\widetilde f\restriction_G=f$. If $(H,\sigma)$ is Hausdorff, then $(G,\tau)$ is closed in $(\widetilde G,\widetilde\tau)$.
\end{theorem}
\begin{proof}
Define $\widetilde G=G\times H$ and $\widetilde\tau=\tau\times\sigma$. It follows from \cite[Remark 2.12]{DGM} that $(G,\tau)$ is topologically isomorphic to $(\Gamma_f,\widetilde\tau\restriction_{\Gamma_f})$ (the topological isomorphism is $j:G\to\Gamma_f$ defined by $j(x)=(x,f(x))$ for every $x\in G$). The canonical projection $p_2:(\widetilde G,\widetilde\tau)\to (H,\sigma)$ is an open continuous homomorphism extending $f$ and so we take $\widetilde f=p_2$. If $(H,\sigma)$ is Hausdorff, $(\Gamma_f,\widetilde\tau\restriction_{\Gamma_f})$ is a closed subgroup of $(\widetilde G,\widetilde\tau)$ by the closed graph theorem.
\end{proof}

Inspired by this result we give the counterpart of \cite[Definition 2]{Ar} for continuous surjective homomorphisms:

\begin{deff}\label{semitopdef}
A continuous surjective homomorphism $f:(G,\tau)\to (H,\sigma)$ of topological groups is \emph{semitopological} if there exist a topological group $(\widetilde G,\widetilde\tau)$ containing $G$ as a topological normal subgroup and an open continuous homomorphism $\widetilde f:(\widetilde G,\widetilde\tau)\to (H,\sigma)$ extending $f$.
\end{deff}

It is obvious that all open continuous surjective group homomorphisms are semitopological.

\medskip
Theorem \ref{ext} shows that every continuous surjective homomorphism of topological groups $f:G\to H$ is the restriction of an open continuous surjective homomorphism of topological groups $\widetilde f:\widetilde G\to H$, where $G$ is a topological subgroup of $\widetilde G$. As noted before not every continuous surjective group homomorphism is the restriction of an open continuous surjective group homomorphism to a normal topological subgroup of the domain, i.e. not all continuous surjective group homomorphisms are semitopological. Consequently the characterization of semitopological homomorphisms can also be viewed as the study of the restrictions of open continuous surjective homomorphisms of topological groups $\widetilde f:\widetilde G\to H$ to normal subgroups $G$ of $\widetilde G$ such that $\widetilde f(G)=H$. From this point of view semitopological homomorphisms are strictly related to the open mapping theorem and its generalizations, which are studied by a lot of authors.

\medskip
Let $\mathcal S$ be the class of all semitopological homomorphisms and $\mathcal S_i$ the class of all semitopological isomorphisms. Obviously $\mathcal S_i\subseteq \mathcal S$.

\medskip
To formulate the main theorem of \cite{Ar} characterizing semitopological isomorphisms, we recall the following concept: for a neighborhood $U$ of the neutral element $e_G$ of a topological group $G$ call a subset $M$ of $G$ \emph{(left) $U$-thin} if $\bigcap\{x^{-1}U x:x\in M\}$ is still a neighborhood of $e_G$. We give some properties of $U$-thin sets in \S\ref{Uthin}.

For a topological group $(G,\tau)$ we denote by $\mathcal V_{(G,\tau)}(e_G)$ the filter of all neighborhoods of $e_G$ in $(G,\tau)$.

\begin{theorem}\label{semitop}\emph{\cite[Theorem 4]{Ar}}
Let $(G,\tau),(H,\sigma)$ be topological groups and $f:(G,\tau)\to(H,\sigma)$ a continuous isomorphism. Then $f$ is semitopological if and only if for every $U\in\mathcal V _{(G,\tau)}(e_G)$
\begin{itemize}
\item[(a)]there exists $V\in\mathcal V_{(H,\sigma)}(e_H)$ such that $f^{-1}(V)$ is $U$-thin;
\item[(b)]for every $g\in G$ there exists $V_g\in\mathcal V_{(H,\sigma)}(e_H)$ such that $[g,f^{-1}(V_g)]\subseteq U$.
\end{itemize}
\end{theorem}

In this case of continuous isomorphisms it is possible to consider without loss of generality the same group $G$ as domain and codomain, endowed with two different group topologies $\tau\geq\sigma$ and as continuous isomorphism the identity map $1_G$ of $G$. In fact, if $f:(G,\tau)\to (H,\eta)$ is a continuous isomorphism of topological groups, then the topology $\sigma=f^{-1}(\eta)$ on $G$ is coarser than $\tau$ and so $1_G:(G,\tau)\to (G,\sigma)$ is a continuous isomorphism and $(G,\sigma)$ is topologically isomorphic to $(H,\eta)$. In particular $1_G:(G,\tau)\to(G,\sigma)$ is semitopological if and only if $f:(G,\tau)\to(H,\eta)$ is semitopological. So the following is an equivalent form of Theorem \ref{semitop}.

\begin{theorem}
Let $G$ be a group and $\tau\geq\sigma$ group topologies on $G$. Then $1_G:(G,\tau)\to(G,\sigma)$ is semitopological if and only if for every $U\in\mathcal V _{(G,\tau)}(e_G)$
\begin{itemize}
\item[(a)]there exists $V\in\mathcal V_{(G,\sigma)}(e_G)$ such that $V$ is $U$-thin;
\item[(b)]for every $g\in G$ there exists $V_g\in\mathcal V_{(G,\sigma)}(e_G)$ such that $[g,V_g]\subseteq U$.
\end{itemize}
\end{theorem}

\medskip
The aim of this paper is to generalize these and most of the remaining results of \cite{Ar} to semitopological homomorphisms. To find a complete characterization of semitopological homomorphisms has turned out to be a non-trivial problem.

\bigskip
However we find a complete characterization in a particular case, which is interesting on its own: as said before semitopological homomorphisms can also be viewed as restrictions $f:G\to H$ of open continuous surjective homomorphisms of topological groups $\widetilde f:\widetilde G\to H$ where $G$ is a normal subgroup of $\widetilde G$ and $f(G)=H$. In this setting one can ask $G$ to be a \emph{dense} normal subgroup of $\widetilde G$:

\begin{deff}
Let $(G,\tau),(H,\sigma)$ be topological groups. A continuous surjective homomorphism $f:(G,\tau)\to(H,\sigma)$ is \emph{d-semitopological} if there exist a topological group $(\widetilde G,\widetilde\tau)$ containing $G$ as a dense normal topological subgroup and an open continuous homomorphism $\widetilde f:(\widetilde G,\widetilde\tau)\to(H,\sigma)$ extending $f$.
\end{deff}

In \S \ref{dense_extensions} we study d-semitopological homomorphisms. In this case considering Hausdorff group topologies simplify the problem and in fact, with the additional assumption that the codomain is complete, we find a complete characterization of d-semitopological homomorphisms in Theorem \ref{d-semitop}. For a subgroup $H$ of a group $G$, $N_G(H)=\{x\in G:x H=H x\}$ is the normalizer of $H$ in $G$. If $G$ is a Hausdorff topological group, we denote by $\overline G$ the two-sided completion of $G$.

\begin{theorem}\label{d-semitop}
Let $f:G\to H$ be a continuous surjective homomorphism, where $G$ is a Hausdorff group and $H$ is a complete group, and let $\overline f:\overline G\to \overline H=H$ be the extension of $f$ to the completions. Then the following conditions are equivalent:
\begin{itemize}
\item[(a)]$f$ is d-semitopological;
\item[(b)]$\overline f\restriction_{N_{\overline G}(G)}:N_{\overline G}(G)\to H$ is open;
\item[(c)]$N_{\overline G}(G)\cap \ker \overline f$ is dense in $\ker\overline f$.
\end{itemize}
\end{theorem}

\bigskip
In \S \ref{characterization} we find necessary and sufficient conditions for a continuous surjective group homomorphism to be semitopological. To do this we introduce new properties. For example we define A-open and strongly A$^*$-open surjective group homomorphisms (see Definitions \ref{A-open_def} and \ref{A*-open_def} respectively) and without any recourse to Arnautov's main theorem (i.e. Theorem \ref{semitop}) we prove that \begin{center}strongly A$^*$-open $\Rightarrow$ semitopological $\Rightarrow$ A-open\end{center} for continuous surjective group homomorphisms (see Theorems \ref{semitophom} and \ref{c*-section}). Since strongly A$^*$-open coincides with A-open for continuous group isomorphisms as well as with conditions (a) and (b) of Theorem \ref{semitop}, Theorem \ref{semitop} is an immediate corollary of this result.
In \S \ref{characterization} we define also A$^*$-open and strongly A-open surjective group homomorphisms (see Definitions \ref{A*-open_def} and \ref{strongly_A-open_def} respectively). Also these properties are equivalent to semitopological for continuous group isomorphisms. The main relations among new and already defined properties for continuous surjective group homomorphisms are the following:
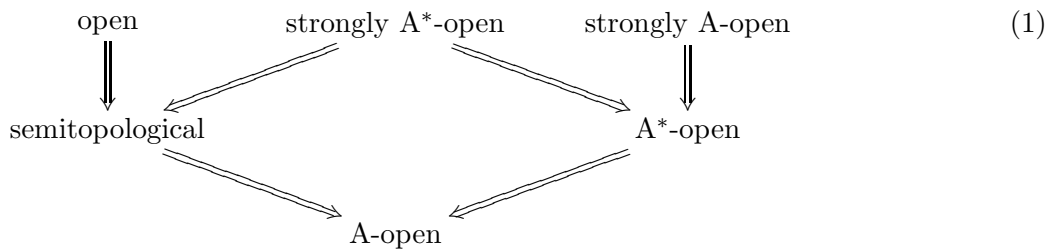
\begin{equation}\label{first_diagram}
\xymatrix{
\text{open} \ar@{=>}[d] & \text{strongly A$^*$-open} \ar@{=>}[dr] \ar@{=>}[dl] & \text{strongly A-open} \ar@{=>}[d]\\
\text{semitopological} \ar@{=>}[dr] & & \text{A$^*$-open} \ar@{=>}[dl] \\
& \text{A-open} &
}
\end{equation}
Moreover in Theorem \ref{initial->all_equivalent} we prove that all these four properties are equivalent to semitopological for a continuous surjective homomorphism of topological groups $f:G\to H$ such that $\ker f$ is contained in the closure of $e_G$ in $G$. This assumption is quite strong, but continuous isomorphisms satisfy it and so we find again Theorem \ref{semitop} as a trivial corollary.

In order to find some property equivalent to semitopological, we analyze the relations among all the properties in (\ref{first_diagram}). Some questions about these relations remain open and a positive answer to them would bring a more precise description of semitopological homomorphisms and their properties.

\medskip
As a particular case of Theorem \ref{semitop}, a continuous isomorphism of topological groups $f:(G,\delta_G)\to(H,\sigma)$, where $\delta_G$ is the discrete topology on $G$, is semitopological if and only if the centralizer $c_H(h)=\{x\in H:[x,h]=e_H\}$ of $h$ in $H$ is $\sigma$-open for every $h\in H$ \cite[Corollary 5]{Ar} (see Corollary \ref{discrhom}). In \S \ref{particular} we consider this particular case for homomorphisms, but we also prove some results in case the topology on the codomain is indiscrete.

Moreover we consider continuous surjective group homomorphisms with the properties in (\ref{first_diagram}) in case the topologies on the domain and on the codomain are very close to be discrete or indiscrete.

In some of these particular cases we find examples and counterexamples clarifying the relations among the properties in (\ref{first_diagram}). This can give an indication of which direction we have to take.

\medskip
Thanks to the characterization of Theorem \ref{semitop} Arnautov proved that $\mathcal S_i$ is stable under taking subgroups, quotients and products:

\begin{theorem}\label{subgroupst}\label{quotientst}
Let $(G,\tau),(H,\sigma)$ be topological groups, $f:(G,\tau)\to(H,\sigma)$ a semitopological isomorphism and $A$ a subgroup of $G$. Then:
\begin{itemize}
	\item[(a)]\emph{\cite[Theorem 7]{Ar}} $f\restriction_A:(A,\tau\restriction_A)\to (f(A),\sigma\restriction_{f(A)})$ is semitopological;
	\item[(b)]\emph{\cite[Theorem 8]{Ar}} in case $A$ is normal and $q:G\to G/A$, $q':H\to H/f(A)$ are the canonical projections, $f':(G/A,\tau_q)\to (H/f(A),\sigma_q)$, defined by $f'(q(g))=q'(f(g))$ for every $g\in G$, is semitopological.
\end{itemize}
\end{theorem}

\begin{theorem}\emph{\cite[Theorem 9]{Ar}}\label{productst}
Let $\{(G_i,\tau_i):i\in I\}$, $\{(H_i,\sigma_i):i\in I\}$ be families of topological groups and $f_i:(G,\tau_i)\to(H,\sigma_i)$ a semitopological isomorphism for every $i\in I$. Then $\prod_{i\in I}f_i:\prod_{i\in I}(G,\tau_i)\to\prod_{i\in I}(H,\sigma_i)$ is semitopological.
\end{theorem}

In \S \ref{stability} we prove (categorical) stability properties of the larger class $\mathcal S$ and of the new properties we introduce, extending substantially Theorems \ref{subgroupst} and \ref{productst}. In particular in Theorem \ref{productst_hom}(a) we prove that also $\mathcal S$ is stable under taking products:
\begin{quote}
\emph{If $\{f_i:i\in I\}$ if a family of continuous surjective group homomorphisms, then $\prod_{i\in I}f_i\in\mathcal S$ if and only if $f_i\in\mathcal S$ for all $i\in I$.}
\end{quote}
For (strongly) A-open and (strongly) A$^*$-open we prove that also the converse implication holds. Consequently, since they coincide with semitopological for continuous group isomorphisms, we obtain as a bonus a new result, namely that the converse implication in Theorem \ref{productst} holds true.

About categorical properties, we prove that $\mathcal S_i$ is closed under pullback along continuous injective homomorphisms and that the class of all A-open continuous surjective homomorphisms is stable under pushout with respect to open continuous surjective group homomorphisms (see Theorems \ref{pullback} and \ref{pushout} respectively). In these terms Theorem \ref{subgroupst} says that $\mathcal S_i$ is closed under pullback with respect to topological embeddings and under pushout with respect to open continuous surjective homomorphisms and so it becomes a corollary of our results.

As noted in \cite{Ar} $\mathcal S_i$ is not closed under taking compositions (see Example \ref{almdiscrex}). Hence also (strongly) A-open and (strongly) A$^*$-open are not stable under taking compositions because they are equivalent to semitopological for continuous group isomorphisms. However we prove some results about composition of semitopological homomorphisms and about composition of group homomorphisms which are (strongly) A-open and (strongly) A$^*$-open.
In Corollary \ref{Aop=Aop=semitop} we find a characterization for an A-open continuous surjective group homomorphism in terms of its algebraically associated isomorphism:
\begin{quote}
\emph{If $f:(G,\tau)\to(H,\sigma)$ is a continuous surjective homomorphism and $q:G\to G/\ker f$ the canonical projection, then by the first homomorphism theorem for topological groups $f':(G/\ker f,\tau_q)\to (H,\sigma)$, defined by $f'(q(g))=f(g)$ for every $g\in G$, is a continuous isomorphism. Then $f$ is A-open if and only if $f'$ is semitopological.}
\end{quote}
Note that the first homomorphism theorem for topological groups says that in this case $f$ is open if and only if $f'$ is open. In analogy our result states that $f$ is A-open if and only if $f'$ is A-open (for continuous group isomorphisms A-open coincides with semitopological).

One of the most interesting results is Proposition \ref{strAop-left_canc} which shows that the class of strongly A-open continuous surjective group homomorphisms is left cancelable. As a consequence we get that the class $\mathcal S_i$ is left cancelable (see Theorem \ref{left-canc}).

\medskip
In \S \ref{questions} we sum up in a more detailed diagram the relations among the properties in (\ref{first_diagram}) for continuous surjective group homomorphisms. We also collect there all open questions. Moreover we remind the open problems left in \cite{Ar} about semitopological isomorphisms, that we discuss in a different paper \cite{DG2}.

\subsubsection*{Notation and terminology}
Let $G$ be a group and $x,y\in G$. We denote by $[x,y]$ the commutator of $x$ and $y$ in $G$, that is $[x,y]=x y x^{-1}y^{-1}$. If $H$ and $K$ are subsets of $G$, let $[H,K]=\langle [h,k]:h\in H, k\in K\rangle$. In case $H=\{x\}$ for some $x\in G$, we write $[x,K]$. Moreover $G'=[G,G]$ is the subgroup of $G$ generated by all commutators of elements of $G$. We denote the center of $G$ by $Z(G)$. 

The function $\Delta:G\to G\times G$ is defined by $\Delta(g)=(g,g)$ for every $g\in G$. If $H$ is another group, we call $p_1:G\times H\to G$ and $p_2:G\times H\to H$ the canonical projections on the first and the second component respectively.

Let $\tau$ be a group topology on $G$. 
If $X$ is a subset of $G$, $\overline X ^\tau$ stands for the closure of $X$ in $(G,\tau)$. If $N$ is a normal subgroup of $G$ and $q:G\to G/N$ is the canonical projection, then $\tau_q$ is the quotient topology of $\tau$ on $G/N$. Moreover $N_\tau$ denotes the subgroup $\overline{\{e_G\}}^\tau$. The discrete topology on $G$ is $\delta_G$ and the indiscrete topology on $G$ is $\iota_G$.
For a Hausdorff topological group $G$ we denote by $\overline G$ the completion of $G$. A group $G$ is complete if $G=\overline G$. Let us note that in particular every complete group is Hausdorff.

For undefined terms see \cite{E,Fuchs}.

\section{Dense extensions}\label{dense_extensions}

We discuss here the problem of characterizing d-semitopological group homomorphisms in the class of Hausdorff topological groups.

\begin{deff}
Let $(G,\tau),(H,\sigma)$ be topological groups and $f:(G,\tau)\to(H,\sigma)$ a continuous surjective homomorphism. Call the topological group $(\widetilde G,\widetilde\tau)$ a \emph{d-extension} of $f:(G,\tau)\to(H,\sigma)$ if $(\widetilde G,\widetilde\tau)$ contains $(G,\tau)$ as a dense normal topological subgroup and there exists an open continuous homomorphism $\widetilde f:(\widetilde G,\widetilde\tau)\to (H,\sigma)$ extending $f$. The homomorphism $\widetilde f$ is called \emph{associated} to the d-extension $(\widetilde G,\widetilde\tau)$.
\end{deff}

Let $(G,\tau)$, $(H,\sigma)$ be Hausdorff groups. Every continuous surjective homomorphism $f:(G,\tau)\to(H,\sigma)$ can be extended to a continuous homomorphism of the completions $\overline f:(\overline G,\overline\tau)\to(\overline H,\overline\sigma)$. If $f$ is open, then $\overline f$ is open.

\medskip
The following fact is due to Grant.

\begin{fact}\label{open_mapping_2}\emph{\cite[Lemma 4.3.2]{DPS}}
Let $G,H$ be Hausdorff topological groups, $f:G\to H$ an open continuous surjective homomorphism and $N$ a dense subgroup of $G$. Then $f\restriction_N:N\to f(N)$ is open if and only if $N\cap \ker f$ is dense in $\ker f$.
\end{fact}

\begin{claim}\label{d-semitop_claim}
Let $G,H$ be Hausdorff topological groups, $f:G\to H$ a continuous surjective homomorphism and let $\overline f:\overline G\to \overline H$ be the extension of $f$ to the completions. If $f$ is d-semitopological, then $N_{\overline G}(G)\cap \ker \overline f$ is dense in $\ker \overline f$.
\end{claim}
\begin{proof}
Let $\widetilde G$ be a topological group such that $G$ is a dense normal subgroup of $\widetilde G$ and $\widetilde f:\widetilde G\to H$ an open continuous surjective homomorphism. Being $G$ dense in $\widetilde G$, by the universal property of the completion we can identify $\widetilde G$ with a subgroup of $\overline G$; in particular $G$ and $\widetilde G$ have the same completion $\overline G$.
Moreover, since $G$ is normal in $\widetilde G$, $\widetilde G\leq N_{\overline G}(G)$.
For the property of the completion, there exists a continuous homomorphism $\overline f:\overline G\to H$ extending $\widetilde f$ (and so extending $f$). Being $\widetilde f$ open and $\widetilde G$ dense in $\overline G$, $H$ dense in $\overline H$, so $\overline f$ is open as well. By Fact \ref{open_mapping_2} $N_{\overline G}(G)\cap \ker \overline f$ is dense in $\ker \overline f$.
\end{proof}

Applying this claim we can now prove Theorem \ref{d-semitop}:

\begin{proof}[\bf Proof of Theorem \ref{d-semitop}]
(a)$\Rightarrow$(c) is Claim \ref{d-semitop_claim} and (b)$\Rightarrow$(a) is obvious because $G$ is dense in $\overline G$ and $G\triangleleft N_{\overline G}(G)$.

\smallskip
(b)$\Leftrightarrow$(c) follows from Fact \ref{open_mapping_2} since $H=\overline H$.
\end{proof}

\begin{corollary}
If there exists a d-extension of a continuous surjective homomorphism of topological groups $f:G\to H$, where $G$ is Hausdorff and $H$ complete, then there exists a maximal one, namely $N_{\overline G}(G)$ with the topology inherited from $\overline G$.
\end{corollary}

Theorem \ref{d-semitop} can be written in the following equivalent form:

\smallskip
\noindent\emph{Let $K,H$ be complete groups and $h:K\to H$ a continuous open surjective homomorphism. If $G$ is a dense subgroup of $K$ such that $h(G)=H$, then the following are equivalent:}
\begin{itemize}
\item[\emph{(a)}]\emph{$h\restriction_G:G\to H$ is d-semitopological;}
\item[\emph{(b)}]\emph{$h\restriction_{N_K(G)}:N_K(G)\to H$ is open;}
\item[\emph{(c)}]\emph{$N_K(G)\cap \ker h$ is dense in $\ker h$.}
\end{itemize}

\section{Thin sets}\label{Uthin}

In the introduction we have defined $U$-thin sets $M$ of a topological group $G$, where $U$ is a neighborhood of $e_G$ in $G$. The subsets $M$ of $G$ that are $U$-thin for every $U$ are precisely the \emph{thin sets} in the sense of Tkachenko \cite{T,T1}. For example compact sets are thin.

The concept of $U$-thin subset of a topological group $G$ is not symmetric and it is possible to give the symmetric definition of \emph{right $U$-thin} set. Note that for every symmetric subset (for example for every subgroup) of $G$ the two concepts coincide. Moreover $M$ is a left $U$-thin subset if and only if $M^{-1}$ is a right $U$-thin subset.
Note that $M$ is $U$-thin if and only if there exists a neighborhood $U_1$ of $e_G$ in $G$ such that $x U_1 x\amu\subseteq U$ for every $x\in M$.

\begin{lemma}\label{thin}
Let $G$ be a topological group, $U,V$ neighborhoods of $e_G$ in $G$ and $M\subseteq G$. Then:
\begin{itemize}
  \item[(a)]if $M$ is $U$-thin, then every $N\subseteq M$ is $U$-thin;
  \item[(b)]if $M_1$ and $M_2$ are $U$-thin, then $M_1\cup M_2$ is $U$-thin;
  \item[(c)]if $M$ is $U$-thin, then $M g$  and $g M$ are $U$-thin for every $g\in G$;
  \item[(d)]if $V\subseteq U$ and $M$ is $V$-thin, then $M$ is $U$-thin;
  \item[(e)]if $M$ is $U$-thin and $V$-thin, then $M$ is $U\cap V$-thin;
  \item[(f)]if $f:G\to H$ is a continuous surjective homomorphism and $W$ is a neighborhood of $e_H$ in $H$, then $f\amu(M)$ is $f\amu(W)$-thin whenever $M\subseteq H$ is $W$-thin.
\end{itemize}
\end{lemma}

Consequently, for a topological group $G$ and a neighborhood $U$ of $e_G$ in $G$, the family $\mathcal I_U(G)=\{M\subseteq G:M\ \text{is $U$-thin}\}$ is a translations invariant ideal.
Observe that $Z(G)\in \mathcal I_U(G)$ for every $U\in\mathcal V_{(G,\tau)}(e_G)$, that is $Z(G)\in\mathcal I(G)=\bigcap\{\mathcal I_U(G):U\in\mathcal V_{(G,\tau)}(e_G)\}=\{M\subseteq G:M\ \text{is thin}\}$.

\begin{lemma}\label{thin_quotient}
Let $G$ be a topological group, $N$ a normal subgroup of $G$ and $q:G\to G/N$ the canonical projection. Moreover let $U$ be a neighborhood of $e_G$ in $G$ and $M\subseteq G$.
\begin{itemize}
\item[(a)]If $M$ is $U$-thin, then $q(M)$ is $q(U)$-thin.
\item[(b)]If $M$ is thin, then $q(M)$ is thin.
\end{itemize}
\end{lemma}

\begin{deff}
A topological group $(G,\tau)$ \emph{has small invariant neighborhoods} (i.e. \emph{$G$ is SIN}) if for every $U\in\mathcal V_{(G,\tau)}(e_G)$ there exists $U'\in\mathcal V_{(G,\tau)}(e_G)$ such that $U'\subseteq U$ and $g U'g^{-1}= U'$ for every $g\in G$ (i.e. $U'$ is \emph{invariant}).
\end{deff}

If $\tau$ is a linear group topology on a group $G$, then $(G,\tau)$ is SIN.
A group $G$ is SIN if and only if $G$ is thin.
The next lemma shows that for SIN groups the characterization given by Theorem \ref{semitop} is simpler, because condition (a) is always verified since a SIN group $G$ is thin, and so only condition (b) remains.

\begin{lemma}\label{SINsemitop}
Suppose that the topological group $(G,\tau)$ is SIN. Then a continuous isomorphism $f:(G,\tau)\to(H,\sigma)$, where $(H,\sigma)$ is a topological group, is semitopological if and only if for every $U\in\VG$ and for every $g\in G$ there exists $V_g\in\VH$ such that $[g,f\amu(V_g)]\subseteq U$.
\end{lemma}

\section{Internal approximations for semitopological homomorphisms}\label{characterization}

Let $f:(G,\tau)\to(H,\sigma)$ be a continuous surjective homomorphism of topological groups. Inspired by Theorem \ref{semitop} we try to find some condition internal to the groups $G$ and $H$ to describe when $f$ is semitopological. The advantage of these internal conditions is obvious. To do this we introduce some concepts which turn out to be equivalent to semitopological for continuous isomorphisms.

\medskip
First we give some definitions and basic properties.

\begin{deff}
Let $(G,\tau),(H,\sigma)$ be topological groups and $f:(G,\tau)\to(H,\sigma)$ a continuous surjective homomorphism. Call the topological group $(\widetilde G,\widetilde\tau)$ an \emph{A-extension} of $f:(G,\tau)\to(H,\sigma)$ if $(\widetilde G,\widetilde\tau)$ contains $(G,\tau)$ as a normal topological subgroup and there exists an open continuous homomorphism $\widetilde f:(\widetilde G,\widetilde\tau)\to (H,\sigma)$ extending $f$. The homomorphism $\widetilde f$ is called \emph{associated} to the A-extension $(\widetilde G,\widetilde\tau)$.
\end{deff}

In the next proposition we give the first properties of the homomorphism associated to an A-extension of a semitopological homomorphism. This proposition generalizes \cite[Proposition 3]{Ar} to semitopological homomorphisms.

\begin{proposition}\label{firstprop}
Let $(G,\tau),(H,\sigma)$ be topological groups and $f:(G,\tau)\to (H,\sigma)$ a semitopological homomorphism. If $(\widetilde G,\widetilde\tau)$ with $\widetilde f:(\widetilde G,\widetilde\tau)\to(H,\sigma)$ is an A-extension of $f$, then:
\begin{itemize}
\item[(a)]$G\cap\ker\widetilde f=\ker f$;
\item[(b)]$\widetilde G=G\cdot\ker\widetilde f$;
\item[(c)]$[\ker\widetilde f,G]\subseteq\ker f$;
\item[(d)]$\ker f$ is a normal subgroup of $\widetilde G$;
\item[(e)]$\widetilde G/\ker f$ is isomorphic to $G/\ker f\times \ker\widetilde f/\ker f$.
\end{itemize}
\end{proposition}
\begin{proof}
(a) and (b) are obvious, (c) follows from (a) and (b) because $G$ is a normal subgroup of $\widetilde G$, (d) follows from (a) and (e) follows from (a), (b) and (d).
\end{proof}

Obviously a continuous surjective homomorphism is semitopological if and only if it has an A-extension.

In the proof of Theorem \ref{semitop} an A-extension of a continuous isomorphism of topological groups $f:(G,\tau)\to(H,\sigma)$ is constructed taking $\widetilde G=G\times H$ and $\{W(U,V):U\in\VG,V\in\VH\}$, where $W(U,V)=\{(u f^{-1}(v),v):u\in U,v\in V\}$, as a base of $\mathcal V_{(\widetilde G,\widetilde\tau)}(e_G)$. Moreover $\widetilde f:\widetilde G\to H$ is defined by $\widetilde f(g,h)=f(g)$ for every $g\in G$. We do an analogous construction in Definition \ref{alpha_s} and use it in the proof of Theorem \ref{c*-section}.

\begin{deff}\label{section}
Suppose that $G$, $H$ are groups and $f:G\to H$ is a surjective homomorphism. 
A \emph{section} is an injective function $s:H\to G$ such that $f(s(h))=h$ for every $h\in H$ and $s(e_H)=e_G.$
\end{deff}

For a surjective group homomorphism there exists always a section.

\begin{remark}
Let $G$, $H$ be groups and $f:G\to H$ a surjective homomorphism. Let $s:H\to G$ be a section of $f$. Then $G=s(H)\cdot\ker f$ and $s(H)\cap\ker f=\{e_G\}$. Moreover, if $X$ is a subset of $H$, then $f\amu (X)=s(X)\cdot\ker f$.
\end{remark}

\begin{deff}\label{alpha_s}
Let $(G,\tau)$, $(H,\sigma)$ be topological groups, $f:(G,\tau)\to (H,\sigma)$ a continuous surjective homomorphism and $s$ a section of $f$. Define the filter base $$\mathcal F_{s,(\tau,\sigma)}=\{W_s(U,V):U\in \VG,V\in \VH \},$$ where $$W_s(U,V)=\{(u s(v),v):u\in U, v\in V\}.$$ In case $\mathcal F_{s,(\tau,\sigma)}$ is a base of the neighborhoods of $(e_G,e_H)$ in $G\times H$ of a group topology on $G\times H$, then we denote this group topology by $\alpha_s(\tau,\sigma)$.
\end{deff}

\begin{lemma}\label{alpha_group-topology}
Let $f:(G,\tau)\to (H,\sigma)$ be a continuous surjective homomorphism of topological groups, $N$ a $\sigma$-open subgroup of $H$ and $s$ a section of $f$. If $\mathcal F_{s,(\tau,\sigma\restriction_N)}$ is a base of the neighborhoods of $(e_G,e_H)$ in $G\times N$ of a group topology $\alpha_s(\tau,\sigma\restriction_N)$ on $G\times N$, then $(G\times N,\alpha_s(\tau,\sigma\restriction_N))$, with $\widetilde f:G\times N\to H$ defined by $\widetilde f(g,n)=f(g)$ for every $(g,n)\in G\times N$, is an A-extension of $f$. In particular $f$ is semitopological.
\end{lemma}
\begin{proof}
Consider $G\times N$ endowed with the group topology $\alpha_s(\tau,\sigma\restriction_N)$. Then $i:G\to G\times N$, defined by $i(g)=(g,e_H)$ for every $g\in G$, is a topological embedding; indeed $W_s(U,V)\cap (G\times \{e_H\})=U\times\{e_H\}$ for every $W_s(U,V)\in\mathcal F_{s,(\tau,\sigma\restriction_N)}$.
Let $V\in\mathcal V_{(H,\sigma)}(e_H)$. There exists $V'\in\mathcal V_{(N,\sigma\restriction_N)}(e_H)$ such that $V'V'\subseteq V$. Since $f:(G,\tau)\to(H,\sigma)$ is continuous, there exists $U\in\VG$ such that $f(U)\subseteq V'$. Then $\widetilde f(W_s(U,V'))=f(U)V'\subseteq V'V'\subseteq V$ and so $\widetilde f$ is continuous. If $W_s(U,V)\in\mathcal F_{s,(\tau,\sigma\restriction_N)}$, then $\widetilde f(W_s(U,V))=f(U)V\subseteq V$ and so $\widetilde f$ is also open.
\end{proof}

In the introduction we have used the graph of a surjective homomorphism of topological groups $f:(G,\tau)\to(H,\sigma)$ to find a topological group $(\widetilde G,\widetilde\tau)$ of which $(G,\tau)$ is a topological subgroup and an open continuous homomorphism $\widetilde f:(\widetilde G,\widetilde \tau)\to (H,\sigma)$ extending $f$: in fact we have taken $\widetilde G=G\times H$ with the product topology and $\widetilde f=p_2$, since $\Gamma_f$ endowed with the topology inherited from $G\times H$ is topologically isomorphic to $(G,\tau)$. But the graph is not helpful when we ask $G$ to be also normal in $\widetilde G$, because the graph is a normal subgroup of $G\times H$ in a very particular case, as we show in the next remark. Anyway it helps in finding a first sufficient condition for a continuous surjective homomorphism to be semitopological (see Corollary \ref{graph}).

\begin{remark}\label{abelian}
Let $G,H$ be groups and $f:G\to H$ a homomorphism. If $H$ is abelian, then $\Gamma_f$ is a normal subgroup of $G\times H$. If $f$ is surjective, then $\Gamma_f$ is a normal subgroup of $G\times H$ if and only if $H$ is abelian.
\end{remark}

Theorem \ref{ext} and Remark \ref{abelian} imply the following corollary that gives a sufficient condition for a continuous surjective group homomorphism to be semitopological.

\begin{corollary}\label{graph}
Let $G,H$ be topological groups and $f:G\to H$ a continuous surjective homomorphism. If $H$ is abelian, then $f$ is semitopological.
\end{corollary}

The converse implication does not hold in general:

\begin{example}
Let $H$ be a discrete non-abelian group. Since every continuous surjective homomorphism $f:G\to H$, where $G$ is a topological group, is open, then $f$ is semitopological.
\end{example}

In general every non-open continuous surjective homomorphism of topological groups with abelian codomain shows that semitopological does not imply open.

\subsection{A-open homomorphisms}

Trying to find a characterization of semitopological homomorphisms, we introduce the following concept, which for continuous group isomorphisms is equivalent to semitopological. In the case of continuous surjective group homomorphisms we prove in Theorem \ref{semitophom} that it is a necessary condition.

\begin{deff}\label{A-open_def}
Let $(G,\tau)$ and $(H,\sigma)$ be topological groups. A homomorphism $f:(G,\tau)\to (H,\sigma)$ is \emph{A-open} if for every $U\in\mathcal V_{(G,\tau)}(e_G)$
\begin{itemize}
\item[(a)]there exists $V\in\mathcal V_{(H,\sigma)}(e_H)$ such that $f^{-1}(V)$ is $f^{-1}(f(U))$-thin;
\item[(b)]for every $g\in G$ there exists $V_g\in\mathcal V_{(H,\sigma)}(e_H)$ such that $[g,f^{-1}(V_g)]\subseteq
f^{-1}(f(U))$.
\end{itemize}
\end{deff}

Every open continuous homomorphism $f:(G,\tau)\to (H,\sigma)$ is A-open. In fact if $U\in\mathcal V_{(G,\tau)}(e_G)$, since $f(U)\in\mathcal V_{(H,\sigma)}(e_H)$, then: there exists $V\in\mathcal V_{(H,\sigma)}(e_H)$ such that $V^3\subseteq f(U)$ and so $f^{-1}(V)$ is $f^{-1}(f(U))$-thin by Lemma \ref{thin}(f); for every $g\in G$ there exists $V_g\in\mathcal V_{(H,\sigma)}(e_H)$ such that $[f(g),V_g]\subseteq f(U)$, being $f(g)\mapsto[f(g),x]$ ($x\in H$) continuous, and so $[g,f^{-1}(V_g)]\subseteq f^{-1}(f(U))$.

\begin{remark}\label{semitop+iso=A-open}
Theorem \ref{semitop} says that \emph{a continuous group isomorphism $f:(G,\tau)\to(H,\sigma)$ is semitopological if and only if $f$ is A-open}.
\end{remark}

In the next theorem we prove that the necessity holds also for continuous surjective group homomorphisms, i.e. that semitopological implies A-open for continuous surjective group homomorphisms. Hence from this theorem we get that open implies A-open, as we have proved directly before.

\begin{theorem}\label{semitophom}
Let $(G,\tau),(H,\sigma)$ be topological groups and $f:(G,\tau)\to(H,\sigma)$ a continuous surjective homomorphism. If $f$ is semitopological then $f$ is A-open.
\end{theorem}
\begin{proof}
Let $(\widetilde G,\widetilde\tau)$ with $\widetilde f:\widetilde G\to H$  be an A-extension of $f:(G,\tau)\to(H,\sigma)$. Let $U\in\mathcal V_{(G,\tau)}(e_G)$.  There exists a neighborhood $W$ of $e_G=e_{\widetilde G}$ in $(\widetilde G,\widetilde\tau)$ such that $W\cap G=U$.

There exists a symmetric $W'\in\mathcal V_{(\widetilde G,\widetilde\tau)}(e_G)$ such that $W'W'W'\subseteq W$. Define $\widetilde f(W')=V\in \mathcal V_{(H,\sigma)}(e_H)$ and $U'=W'\cap G\in\mathcal V_{(G,\tau)}(e_G)$. We prove that $f^{-1}(V)$ is $f^{-1}(f(U))$-thin. Equivalently we show that $x U' x^{-1}\subseteq f^{-1}(f(U))$ for every $x\in f^{-1}(V)$. Let $x\in f^{-1}(V)$ and $g\in U'$. Since $f^{-1}(V)= f^{-1}(\widetilde f(W'))\subseteq\widetilde f^{-1}(\widetilde f(W'))=W'\ker\widetilde f$, there exist $\widetilde x\in W$ and $\widetilde b\in \ker\widetilde f$ such that $x=\widetilde x\widetilde b$. Then $$f(x g x^{-1})=f(\widetilde x \widetilde b g \widetilde b^{-1}\widetilde x^{-1})=f(\widetilde x g \widetilde x^{-1})$$ and $\widetilde x g \widetilde x^{-1}\in W\cap G=U$. Hence $f(x g x^{-1}) \in f(U)$ and so $x g x^{-1}\in f^{-1}(f(U))$.

For every $g\in G$ there exists a neighborhood $W'\in\mathcal V_{(\widetilde G,\widetilde\tau)}(e_G)$ of $e_{\widetilde G}$ in $\widetilde G$ such that $[g,W']\subseteq W$ by the continuity of $[g,-]:\widetilde G\to \widetilde G$. Define $V_g=\widetilde f(W')\in\mathcal V_{(H,\sigma)}(e_H)$. We prove that $[g,f^{-1}(V_g)]\subseteq f^{-1}(f(U))$. Let $x\in f^{-1}(V_g)$; as noted previously there exist $\widetilde x\in W'$ and $\widetilde b\in\ker\widetilde f$ such that $x=\widetilde x\widetilde b$. Therefore $$f(g x g^{-1}x^{-1})=f(g\widetilde x\widetilde b g^{-1}\widetilde b^{-1}\widetilde x^{-1})=f(g\widetilde x g^{-1}\widetilde x^{-1}),$$ where $g\widetilde x g^{-1}\widetilde x^{-1}\in W\cap G=U$.
\end{proof}

\subsection{A$^*$-open and strongly A$^*$-open homomorphisms}

Since we have proved only that for continuous surjective group homomorphisms A-open is weaker than semitopological and we would like to have a condition that implies semitopological, in Definition \ref{A*-open_def} we give two properties stronger than A-open.

\begin{deff}
A map $t:H\to G$ between topological groups is a \emph{quasihomomorphism} if $t(e_H)=e_G$ and the maps
\begin{itemize}
	\item $t':H\times H\to G$, defined by $t'(h_1,h_2)=t(h_1)t(h_2)t(h_1 h_2)^{-1}$ for every $h_1,h_2\in H$, and
	\item $t_y:H\to G$ for every $y\in H$, obtained putting $t_y(h)=t(y)t(h)t(y)^{-1}t(y h y^{-1})^{-1}$ for every $h\in H$,
\end{itemize}
are continuous at $(e_H,e_H)$ and $e_H$ respectively.
\end{deff}

This definition is the counterpart of the definition of quasihomomorphism given in \cite{Cab} for abelian groups. In fact, if $H$ and $G$ are abelian, $t_y$ is the identity map for every $y\in H$ and the continuity of $t'$ at $(e_H,e_H)$ is exactly the condition given in \cite{Cab}. 

If a map $t:H\to G$ between topological groups is a homomorphism or simply $t$ is continuous at $e_H$, then $t$ is a quasihomomorphism.

\begin{remark}\label{quasihom=hom}
A map $t:(H,\iota_H)\to (G,\delta_G)$ is a quasihomomorphism if and only if it is a homomorphism. In fact, if $t$ is a quasihomomorphism, then $t'(H,H)=\{e_G\}$ and $t_{e_H}(H)=\{e_G\}$. Therefore $t$ is a homomorphism.
\end{remark}

\begin{claim}\label{initial->quasihom}
Let $f:(G,\tau)\to(H,\sigma)$ be a continuous surjective homomorphism of topological groups. If $\ker f\subseteq N_\tau$, then every section $s$ of $f$ is a quasihomomorphism.
\end{claim}
\begin{proof}
For all $h_1,h_2\in H$ $$f(s'(h_1,h_2))=f(s(h_1)s(h_2)s(h_1h_2)\amu)=f(s(h_1))f(s(h_2))f(s(h_1h_2))\amu=e_H$$ and so $s'(h_1,h_2)\in \ker f$. Since $\ker f\subseteq U$ for each $U\in\VG$, $s'$ is continuous at $e_H$. Analogously let $y\in H$; for every $h\in H$ $$f(s_y(h))=f(s(y)s(h)s(y)\amu s(y h y \amu) \amu) = e_H$$ and so $s_y(h)\in\ker f$. Since $\ker f\subseteq U$ for each $U\in\VG$, $s_y$ is continuous at $e_H$. 
\end{proof}

\begin{deff}\label{A*-open_def}
Let $(G,\tau)$ and $(H,\sigma)$ be topological groups. A surjective homomorphism $f:(G,\tau)\to (H,\sigma)$
is \emph{A$^*$-open} if there exists a section $s$ of $f$ such that for every $U\in\mathcal V_{(G,\tau)}(e_G)$
\begin{itemize}
\item[(a)]there exists $V\in\mathcal V_{(H,\sigma)}(e_H)$ such that $s(V)$ is $U$-thin;
\item[(b)]for every $g\in G$ there exists $V_g\in\mathcal V_{(H,\sigma)}(e_H)$ such that $[g,s(V_g)]\subseteq U$.
\end{itemize}
If $s$ is also a quasihomomorphism, $f$ is \emph{strongly A$^*$-open}.
\end{deff}

Note that $s'(H\times H)\cup\bigcup_{y\in H}s_y(H)\subseteq \ker f$.

\medskip
For continuous group isomorphisms the conditions A$^*$-open and strongly A$^*$-open coincide and they coincide also with A-open. For homomorphisms only one implication holds:

\begin{proposition}\label{A*op->A-op}
Let $(G,\tau),(H,\sigma)$ be topological groups and $f:(G,\tau)\to (H,\sigma)$ a continuous surjective homomorphism. If $f$ is A$^*$-open, then $f$ is A-open.
\end{proposition}
\begin{proof}
There exists a section $s$ of $f$ that witnesses that $f$ is A$^*$-open. Let $U\in\VG$. There exists $V\in\VH$ such that $s(V)$ is $U$-thin. We prove that $f\amu(V)$ is $f\amu(f(U))$-thin. There exists $U'\in\VG$ such that $x U' x\amu\subseteq U$ for every $x\in s(V)$. Then $f\amu(f(x U' x\amu))\subseteq f\amu(f(U))$ for every $x\in s(V)$. For every $y\in f\amu(V)$ there exists $x\in s(V)$ such that $y U' y\amu\subseteq f\amu(f( x U' x\amu))=f\amu(f(x) f(U') f(x)\amu)$. Indeed, if $y\in f\amu(V)$, then $f(y)=v\in V$. Put $x=s(v)\in s(V)$. Thus $f(y)=f(x)$ and so, for every $u\in U'$, $$f(y u y\amu)= f(y) f(u) f(y)\amu=f(x) f(u) f(y)\in f(x) f(U') f(x).$$ This means that $y u y\amu\in f\amu(f(x) f(U') f(x)\amu)=f\amu(f(x U' x\amu))$ for every $u\in U'$. Hence $y U' y\amu\subseteq f\amu(f(U))$ for every $y\in f\amu(V)$.

Let $g\in G$. There exists $V_g\in\VH$ such that $[g,s(V_g)]\subseteq U$. Then $$f\amu([f(g),V_g])=f\amu([f(g),f(s(V_g))])=f\amu(f([g,s(V_g)]))\subseteq f\amu(f(U)).$$ Let $y\in f\amu(V_g)$. Then $f([g,y])=[f(g),f(y)]\in [f(g),V_g]$ and this yields $$[g,f\amu(V_g)]\subseteq f\amu([f(g),V_g])\subseteq f\amu(f(U)),$$ which completes the proof.
\end{proof}

Example \ref{Aop-<A*op} shows that the converse implication of this proposition does not hold in general.
Moreover we don't know whether in general open implies A$^*$-open for continuous surjective homomorphisms of topological groups.
The question is open in general but we have positive answer in some particular cases: for example when the continuous surjective homomorphism of topological groups $f:(G,\tau)\to (H,\sigma)$ is an isomorphism or when $\ker f\subseteq N_\tau$.

\begin{remark}\label{initial}
Let $f:(G,\tau)\to (H,\sigma)$ be a continuous surjective homomorphism. If $\ker f\subseteq N_\tau$, i.e. $\ker f\subseteq U$ for every $U\in\VG$, then $\tau$ is the initial topology of $\tau_q$ on $G/\ker f$. In this situation it is possible to think that $U\in\VG$ is such that $U=q\amu(q(U))$, where $q:G\to G/\ker f$ is the canonical projection. Continuous isomorphisms trivially satisfy this condition.
\end{remark}

Let $f:G\to H$ be a surjective homomorphism. There exists a section $s$ of $f$ which is a homomorphism if and only if $G$ is isomorphic to the semidirect product $\ker f\ltimes H$. In this case, when the topology on $G\cong \ker f\ltimes H$ is the product topology, strongly A$^*$-open obviously coincides with A$^*$-open.

\begin{proposition}\label{open+s-hom->strA*op}
For a continuous surjective homomorphism of topological groups $f:(G,\tau)\to (H,\sigma)$, in case $(G,\tau)\cong(\ker f\ltimes H,\tau\restriction_{\ker f}\times \tau\restriction_H)$, if $f$ is open (i.e. $\tau\restriction_H=\sigma$) then $f$ is strongly A$^*$-open (i.e. A$^*$-open).
\end{proposition}
\begin{proof}
Let $L=\ker f$ and suppose without loss of generality that $G=L\ltimes H$; then $f=p_2$ and $s:H\to G$, defined by $s(h)=(e_L,h)$ for every $h\in H$, is a section of $f$ and a homomorphism. So it suffices to verify that $s$ makes $f$ A$^*$-open. Let $U\in\VG$. There exists $U'=W\times V\in\VG$, where $W\in\mathcal V_{(L,\tau\restriction_L)}(e_L)$ and $V\in\VH$, such that $U'U'U'\subseteq U$. Since $s(V)\subseteq U'$, then $s(v) U's(v)\amu\subseteq U$ for every $v\in V$, that is $s(V)$ is $U$-thin. Let $g\in G$. There exists $U'\in\VG$ such that $U'U'\subseteq U$ and there exists $U_g=W_g\times V_g\in\VG$, where $W_g\in\mathcal V_{(L,\tau\restriction_L)}(e_L)$ and $V_g\in\VH$, such that $U_g\subseteq U'$ and $g U_g g\amu\subseteq U'$. Let $v\in V_g$. Then, since $s(V_g)\subseteq U_g$, $$[g,s(v)]=g s(v) g\amu s(v)\amu \in g s(V_g) g\amu s(V_g)\subseteq g U_g g\amu U_g\subseteq U'U'\subseteq U.$$ Hence $[g,s(V_g)]\subseteq U$ and this completes the proof that $f$ is strongly A$^*$-open.
\end{proof}

Example \ref{c,non-c*} shows that the existence of a section which is a quasihomomorphism of a continuous surjective group homomorphism $f$ does not yield in general that $f$ is A$^*$-open.

\medskip
Theorem \ref{c*-section} is one of the main results of this paper. The technique used in its proof is inspired by that of the proof in \cite{Ar} of Theorem \ref{semitop}.

\begin{theorem}\label{c*-section}
Let $(G,\tau),(H,\sigma)$ be topological groups, $N$ a $\sigma$-open subgroup of $H$, $f:(G,\tau)\to(H,\sigma)$ a continuous surjective homomorphism and $s$ a section of $f\restriction_{f\amu(N)}:f\amu(N)\to N$.
Then $\mathcal F_{s,(\tau,\sigma\restriction_N)}$ is a base of the neighborhoods of $(e_G,e_H)$ in $G\times N$ of a group topology $\alpha_s(\tau,\sigma\restriction_N)$ on $G\times N$ if and only if $s$ makes $f$ strongly A$^*$-open.
\end{theorem}
\begin{proof}
Suppose that $\mathcal F_{s,(\tau,\sigma\restriction_N)}$ (see Definition \ref{alpha_s}) is a base of the neighborhoods of $(e_G,e_H)$ in $G\times N$ of a group topology $\alpha_s(\tau,\sigma\restriction_N)$ on $G\times N$ and let $U\in\VG$. So $W(U,V)=W_s(U,V)\in\mathcal F_{s,(\tau,\sigma\restriction_N)}$ for some $V\in\mathcal V_{(N,\sigma\restriction_N)}(e_H)$. Then there exists $W(U',V')\in\mathcal F_{s,(\tau,\sigma\restriction_N)}$ such that $W(U',V')W(U',V')\subseteq W(U,V)$. Therefore $(u s(v), v)(u' s(v'), v')\in W(U,V)$ for every $u,u'\in U',v,v'\in V'$ and this yields $$(u s(v) u' s(v'),v v')=(u s(v) u' s(v') s(v v')^{-1} s(v v'),v v')\in W(U,V).$$ Choosing $u=u'=e_G$ we get $s(v) s(v') s(v v')^{-1}\in U$ for every $(v,v')\in V' \times V'$  and so $s'$ is continuous at $(e_H,e_H)$.

Let $y\in H$, $x\in G$ and $W(U,V)\in\mathcal F_{s,(\tau,\sigma\restriction_N)}$. Then there exists $W(U',V')\in\mathcal F_{s,(\tau,\sigma\restriction_N)}$ such that $$(x,y)W(U',V')(x,y)^{-1}\subseteq W(U,V).$$ Hence $(x,y)(u s(v),v)(x^{-1},y^{-1})\in W(U,V)$ for every $u\in U',v\in V'$ and so $$W(U,V)\ni (x u s(v) x^{-1},y v y^{-1})=(x u s(v) x^{-1} s(y v y^{-1})^{-1} s(y v y^{-1}),y v y^{-1}).$$ Taking $x=s(y)$ and $u=e_G$ we get in particular $s(y)s(v)s(y)^{-1}s(y v y^{-1})^{-1}\in U$ for every $v\in V'$ and hence $s_y$ is continuous in $e_H$.

This proves that $s$ is a quasihomomorphism.

\smallskip
To prove that $s$ makes $f$ strongly A$^*$-open it remains to prove that $s$ makes $f$ A$^*$-open. Let $U\in\VG$ and $V\in\VH$. Then there exists $U'\in\VG$ such that $U'$ is symmetric and $U'U'\subseteq U$. Let $$s'':H\to G\ \text{defined by} \ s''(h)=s'(h,h\amu)=s(h)s(h\amu)\ \text{for every}\ h\in H.$$ The continuity of $s'$ at $(e_H,e_H)$ implies the continuity of $s''$ at $e_H$ because $s(e_H)=e_G$. Since $s''$ is continuous at $e_H$, so there exists a symmetric $V'\in\VH$ such that $s''(V')\subseteq U$. By the hypothesis there exist $U''\in\VG$ and a symmetric $V''\in\VH$ such that $V''\subseteq V'$ and $W(U'',V'')\amu\subseteq W(U',V)$. Hence $(u s(v),v)\amu=(s(v)\amu u\amu, v\amu)\in W(U',V)$ for every $u\in U'',v\in V''$. Consequently $s(v)\amu u\amu s(v\amu)\amu\in U'$. Since $U'$ is symmetric $s(v\amu) u s(v)\in U'$ for every $u\in U'', v\in V''$. Moreover $s(v) u s(v\amu)\in U'$ for every $u\in U'', v\in V''$, because $V''$ is symmetric. Since $s''(V'')\subseteq s''(V')\subseteq U'$, for every $v\in V''$ there exists $u'\in U'$ such that $s(v\amu)=s(v)\amu u'$. Then for every $u\in U''$, $v\in V''$ there exists $u'\in U'$ such that $$s(v)us(v)\amu=s(v) u s(v\amu) u'\in U'U'\subseteq U.$$ This means that $s(V'')$ is $U$-thin.

Let $g\in G$. By the hypothesis there exist $U'\in\VG$ and $V'\in\VH$ such that $$(g,e_H)W(U',V')(g\amu,e_H)\subseteq W(U,V).$$ Consequently $(g u s(v) g\amu, v)\in W(U,V)$ for every $u\in U',v\in V'$. Then $g u s(v) g\amu s(v)\amu\in U$ for every $u\in U',v\in V'$. Take $u=e_G$; so $[g,s(v)]\in U$ for every $v\in V'$, that is $[g,s(V')]\subseteq U$. Hence $f$ is strongly A$^*$-open.

\medskip
Suppose that $s'$ and $s_y$ (for every $y\in H$) are continuous at $(e_H,e_H)$ and $e_H$ respectively.
We want to prove that the filter base $\mathcal F_{s,(\tau,\sigma\restriction_N)}$ is a base of the neighborhoods of $(e_G,e_H)$ in $G\times N$. Let $W(U,V)\in\mathcal F_{s,(\tau,\sigma\restriction_N)}$.

(1) There exist $U'\in\VG$ and $V'\in\mathcal V_{(N,\sigma\restriction_N)}(e_H)$ such that $U'U'U'\subseteq U$ and $V'V'\subseteq V$. Since $f\restriction_{f^{-1}(N)}:f^{-1}(N)\to N$ is A$^*$-open, there exists $V''\in\mathcal V_{(N,\sigma\restriction_N)}(e_H)$ such that $V''\subseteq V'$ and $s(V'')$ is $U'$-thin; in particular there exists $U''\in\VG$ such that $U''\subseteq U'$ and $x U''x^{-1}\in U'$ for every $x\in s(V'')$. Thanks to the continuity of $s'$ at $(e_H,e_H)$ there exists $V'''\in\mathcal V_{(N,\sigma\restriction_N)}(e_H)$ such that $V'''\subseteq V''$ and $s(v)s(v')\in U'' s(v v')$ for every $v,v'\in V'''$. Then $W(U'',V''')W(U'',V''')\subseteq W(U,V)$: for every $u,u'\in U''$ and $v,v'\in V'''$ there exists $u''\in U''$ such that 
\begin{align*}(u s(v),v)(u' s(v'),v')=(u s(v) u' s(v'),v v')&=(u(s(v) u' s(v)^{-1})s(v) s(v'),v v')=\\
&=(u(s(v) u' s(v)^{-1}) u'' s(v v'),v v')\in W(U,V).
\end{align*}

(2) There exists $U'\in\VG$ such that $U'U'\subseteq U$ and $U'$ is symmetric. As shown before, the continuity of $s'$ at $(e_H,e_H)$ implies the continuity at $e_H$ of the previously defined $s''$. Since $s''$ is continuous at $e_H$, there exists $V'\in\mathcal V_{(N,\sigma\restriction_N)}(e_H)$ such that $V'\subseteq V$ and $s(v)s(v^{-1})\in U'$ for every $v\in V'$. Since $f\restriction_{f^{-1}(N)}:f^{-1}(N)\to N$ is A$^*$-open, there exists $V''\in \mathcal V_{(N,\sigma\restriction_N)}(e_H)$ such that $V''$ is symmetric, $V''\subseteq V'$ and $s(V'')$ is $U'$-thin; in particular there exists $U''\in\VG$ such that $U''$ is symmetric and $x U''x^{-1}\subseteq U'$ for every $x\in s(V'')$. Then $W(U'',V'')^{-1}\subseteq W(U,V)$. In fact for every $u\in U''$ and $v\in V''$ there exists $u'\in U'$ such that 
\begin{align*}(u s(v),v)^{-1}= (s(v)^{-1}u^{-1},v^{-1}) &= (s(v)^{-1}u^{-1}s(v^{-1})^{-1}s(v^{-1}),v^{-1})=\\ &=((s(v)^{-1}u^{-1}s(v)) u' s(v^{-1}),v^{-1}) \in W(U,V).
\end{align*}

(3) Let $(x,y)\in G\times N$. There exist $U'\in\VG$ and $V'\in\mathcal V_{(N,\sigma\restriction_N)}(e_H)$ such that $U'$ is symmetric, $U'U'U'U'\subseteq U$ and $y V'y^{-1}\subseteq V$. There exists also $U''\in\VG$ such that $x U''x^{-1}\subseteq U'$ and $s(y)U''s(y)^{-1}\subseteq U'$. Since $f\restriction_{f^{-1}(N)}:f^{-1}(N)\to N$ is A$^*$-open, there exists a symmetric $V''\in\mathcal V_{(N,\sigma\restriction_N)}(e_H)$ such that $V''\subseteq V'$, $[x,s(V'')]\subseteq U'$ and $[s(y)^{-1},s(V'')]\subseteq U''$. By the continuity of $s_y$ at $e_H$ there exists $V'''\in\mathcal V_{(N,\sigma\restriction_N)}(e_H)$ such that $V'''\subseteq V''$ and $s(y)s(v)s(y)^{-1}\in U' s(y v y^{-1})$ for every $v\in V'''$. Then $(x,y)W(U',V'')(x,y)^{-1}\subseteq W(U,V)$: for every $u\in U''$ and $v\in V'''$ there exists $u'\in U'$ such that 
\begin{align*}
(x,y)(u s(v),v)(x^{-1},y^{-1}) &=(x u s(v) x^{-1},y v y^{-1})=\\
&=((x u x^{-1})(x s(v) x^{-1})s(y v y^{-1})^{-1}s(y v y^{-1}),y v y^{-1})\in W(U,V)
\end{align*}
because $y v y^{-1}\in y V' y{-1}\subseteq V$ and
\begin{align*}
&(x u x^{-1})(x s(v) x^{-1})s(y v y^{-1})^{-1} =(x u x^{-1})(x
s(v) x^{-1})s(y)s(v)^{-1}s(y)^{-1} u'=\\ &=(x u x^{-1})(x s(v)
x^{-1} s(v)^{-1})s(y)(s(y)^{-1}s(v)s(y)s(v)^{-1})s(y)^{-1} u'\in U.
\end{align*}
This completes the proof.
\end{proof}

The next corollary is one of the most interesting results of this paper because it gives a sufficient condition for a continuous surjective homomorphism to be semitopological:

\begin{corollary}\label{c-section_corollary}
Let $f:(G,\tau)\to (H,\sigma)$ be a continuous surjective homomorphism. If $f$ is strongly A$^*$-open, then $f$ is semitopological.
\end{corollary}
\begin{proof}
There exists a section $s$ of $f$ that makes $f$ strongly A$^*$-open. By Theorem \ref{c*-section} $\mathcal F_{s,(\tau,\sigma)}$ is a base of the neighborhoods of $(e_G,e_H)$ in $G\times H$ of a group topology $\alpha_s(\tau,\sigma)$. Then $f$ is semitopological by Lemma \ref{alpha_group-topology}.
\end{proof}

The converse implication does not hold in general --- see Example \ref{strAop-nonquasihomo}.

\smallskip
We have also the following corollary of Theorem \ref{c*-section}.

\begin{corollary}\label{corcor}
Let $(G,\tau),(H,\sigma)$ be topological groups, $N$ a $\sigma$-open subgroup of $H$ and $f:(G,\tau)\to(H,\sigma)$ a continuous surjective homomorphism such that $f\restriction_{f^{-1}(N)}:f^{-1}(N)\to N$ is A$^*$-open. If there exists a section $s$ of $f$ which is a homomorphism, then $f:(G,\tau)\to (H,\sigma)$ is semitopological.
\end{corollary}

As a consequence of Proposition \ref{open+s-hom->strA*op} and Corollary \ref{corcor} we have that for a continuous surjective homomorphism of topological groups $f:(G,\tau)\to(H,\sigma)$, 
if $G\cong \ker f\ltimes H$ and $\tau=\tau\restriction_{\ker f}\times \tau\restriction_H$, in particular there exists a section $s$ of $f$ which is a homomorphism and then $f$ open $\Rightarrow$ $s$ makes $f$ A$^*$-open $\Rightarrow$ $f$ semitopological.

\smallskip
Example \ref{c,non-c*} shows that $s$ homomorphism does not imply that $s$ makes $f$ A$^*$-open.

\begin{remark}
Corollary \ref{c-section_corollary} together with Theorem \ref{semitophom} imply Theorem \ref{semitop}. Indeed let $f:(G,\tau)\to(H,\sigma)$ be a continuous isomorphism of topological groups; Theorem \ref{semitop} states that $f$ is semitopological if and only if $f$ is A-open.  Since $f^{-1}$ is a homomorphism, in particular it is a quasihomomorphism. Since for $f$ continuous isomorphism A$^*$-open is the same as A-open as noted before, Corollary \ref{c-section_corollary} and Theorem \ref{semitophom} prove Theorem \ref{semitop}.
\end{remark}

\subsection{Strongly A-open homomorphisms}

In this section we introduce another property stronger than A$^*$-open:

\begin{deff}\label{strongly_A-open_def}
Let $(G,\tau)$ and $(H,\sigma)$ be topological groups. A homomorphism $f:(G,\tau)\to (H,\sigma)$ is \emph{strongly A-open} if for every $U\in\mathcal V_{(G,\tau)}(e_G)$
\begin{itemize}
\item[(a)]there exists $V\in\mathcal V_{(H,\sigma)}(e_H)$ such that $f^{-1}(V)$ is $U$-thin;
\item[(b)]for every $g\in G$ there exists $V_g\in\mathcal V_{(H,\sigma)}(e_H)$ such that $[g,f^{-1}(V_g)]\subseteq U$.
\end{itemize}
\end{deff}

From the definition we have directly that if $f:(G,\tau)\to (H,\sigma)$ is a strongly A-open continuous surjective group homomorphism, then $$[G,\ker f]\subseteq N_\tau.$$
In particular $\ker f$ is $U$-thin for every $U\in\VG$.

Every strongly A-open homomorphism is A$^*$-open and so also A-open by Proposition \ref{A*op->A-op}. Note that Example \ref{c,non-c*} shows that the existence of a section which is a quasihomomorphism does not imply the condition strongly A-open.

\begin{lemma}\label{Z-trivial}
Let $(G,\tau),(H,\sigma)$ be topological groups and $f:(G,\tau)\to(H,\sigma)$ a continuous surjective homomorphism. Suppose that $\tau$ is Hausdorff and $Z(G)=\{e_G\}$. If $f$ is strongly A-open, then $f$ is a (semitopological) isomorphism.
\end{lemma}
\begin{proof}
Since $\tau$ is Hausdorff, the previous observation implies that $\ker f\subseteq Z(G)$. Since $Z(G)=\{e_G\}$, $\ker f$ is trivial and $f$ is an isomorphism.
\end{proof}

The previous lemma yields that in general semitopological cannot imply strongly A-open.

Moreover we construct the following example, which shows that open together with (strongly) A$^*$-open do not imply strongly A-open. Consequently strongly A-open cannot be equivalent to semitopological in general. However we see in \S\ref{stability} that the class of all strongly A-open continuous surjective homomorphisms has good categorical properties. This allows us to find some nice property of the class $\mathcal S_i$.

\begin{example}\label{open,A*-open,non-strongly-A-open}
Let $K$ be a compact Hausdorff simple group (take for example $K=SO(3,\mathbb R)$). Let $G=K\times K$ and $f=p_1:K\times K\to K$. First of all $f$ is open. Moreover $f$ is not strongly A-open: if it was strongly A-open, since $Z(K)$ is trivial, then it would be an isomorphism by Lemma \ref{Z-trivial}, but $\ker f=\{e_K\}\times K$.
The section $s:K\to G$ of $f$ defined by $s(x)=(x,e_K)$ for every $x\in K$ is a homomorphism. Hence $f$ is strongly A$^*$-open by Proposition \ref{open+s-hom->strA*op}.
\end{example}

\begin{remark}\label{iso->all_equivalent}
For a continuous isomorphism of topological groups $f:(G,\tau)\to(H,\sigma)$ it is equivalent to be semitopological, A-open, A$^*$-open, strongly A$^*$-open and strongly A-open.
\end{remark}

\begin{example}\label{non-open,str_A-open}
In view of Remark \ref{iso->all_equivalent} and Corollary \ref{graph} it suffices to consider a non-open continuous isomorphism of topological abelian groups to see that in general (semitopological and) strongly A-open does not imply open.
\end{example}

The next theorem shows that for a continuous surjective homomorphims such that its kernel is contained in every neighborhood of the neutral element the four new conditions introduced are equivalent to semitopological. Theorem \ref{semitop} is a corollary of the following theorem, since every continuous isomorphism trivially satisfies the assumption, as noted previously in Remark \ref{initial}. As a consequence of Corollary \ref{Aop=Aop=semitop}, in Corollary \ref{initial->all_equivalent_corollary} we add another equivalent condition to those of the next theorem.

\begin{theorem}\label{initial->all_equivalent}
Let $(G,\tau),(H,\sigma)$ be topological groups and $f:(G,\tau)\to(H,\sigma)$ a continuous surjective homomorphism such that $\ker f\subseteq N_\tau$. The following are equivalent:
\begin{itemize}
    \item[(a)]$f$ is strongly A-open;
    \item[(b)]$f$ is strongly A$^*$-open;
    \item[(c)]$f$ is A$^*$-open;
    \item[(d)]$f$ is A-open;
    \item[(e)]$f$ is semitopological.
\end{itemize}
\end{theorem}
\begin{proof}
(a)$\Rightarrow$(c) is trivial and (c)$\Rightarrow$(d) by Proposition \ref{A*op->A-op}.

\smallskip
(d)$\Rightarrow$(a) Let $U_0\in\mathcal V_{(G,\tau)}(e_G)$. Then there exists $U\in\mathcal V_{(G,\tau)}(e_G)$ such that $U U\subseteq U_0$. Let $g\in G$. Since $f$ is A-open there exist $V,V_g\in\mathcal V_{(H,\sigma)}(e_H)$ such that $f^{-1}(V)$ is $f^{-1}(f(U))$-thin and $[g,f^{-1}(V_g)]\subseteq f^{-1}(f(U))$. But $f^{-1}(f(U))=U\ker f\subseteq UU\subseteq U_0$ and so $f^{-1}(V)$ is $U_0$-thin and $[g,f^{-1}(V_g)]\subseteq U_0$.

\smallskip
(e)$\Rightarrow$(d) by Theorem \ref{semitophom} and (c)$\Rightarrow$(e) by Claim \ref{initial->quasihom} and Corollary \ref{c-section_corollary}.

\smallskip
(b)$\Leftrightarrow$(c) by Claim \ref{initial->quasihom}.
\end{proof}

In the next lemma we consider the case when the topology on the domain in SIN. In this case the first condition in the definitions of strongly (A-open) and A$^*$-open is automatically satisfied.

\begin{lemma}\label{SIN_case}
Suppose that the topological group $(G,\tau)$ is SIN and let $f:(G,\tau)\to(H,\sigma)$ be a continuous surjective group homomorphism. Then:
\begin{itemize}
	\item[(a)]$f$ is A-open if and only if for every $U\in\VG$ and for every $g\in G$ there exists $V_g\in\VH$ such that $[g,f\amu(V_g)]\subseteq f\amu(f(U))$;
	\item[(b)]$f$ is A$^*$-open if and only if there exists a section $s$ of $f$ such that for every $U\in\VG$ and for every $g\in G$ there exists $V_g\in\VH$ such that $[g,s(V_g)]\subseteq U$;
	\item[(c)]$f$ is strongly A-open if and only if for every $U\in\VG$ and for every $g\in G$ there exists $V_g\in\VH$ such that $[g,f\amu(V_g)]\subseteq U$.
\end{itemize}
\end{lemma}

\begin{proposition}\label{projection}
Let $(G,\tau)=(L,\rho)\times (H,\sigma')$ ($\tau=\rho\times\sigma'$) be a topological group and $f=p_2:(G,\tau)\to (H,\sigma)$ be the canonical projection, where $\sigma$ is a group topology on $H$ such that $\sigma\leq\sigma'$. Then the following are equivalent:
\begin{itemize}
	\item[(a)]$f$ is A-open;
	\item[(b)]$f$ is A$^*$-open;
	\item[(c)]$f$ is strongly A$^*$-open.
\end{itemize}
Moreover we have the following characterizations:
\begin{itemize}
	\item[(i)] $f$ is A-open if and only if $1_H:(H,\sigma')\to(H,\sigma)$ is semitopological;
	\item[(ii)] $f$ is strongly A-open if and only if $L'\subseteq N_\rho$ and $1_H:(H,\sigma')\to(H,\sigma)$ is semitopological.
\end{itemize}
If $L$ is abelian also the following condition is equivalent to (a)--(c):
\begin{itemize}
	\item[(d)]$f$ is strongly A-open.
\end{itemize}
\end{proposition}
\begin{proof}
(b)$\Rightarrow$(a) by Proposition \ref{A*op->A-op} and (b)$\Leftrightarrow$(c) because $s:H\to G$ such that $s(h)=(e_L,h)$ for all $h\in H$ is a section of $f$ which is a homomorphism.

\smallskip
(i) The homomorphism $f$ is A-open if and only if for every $U=U'\times W\in\VG$, where $U'\in\mathcal V_{(L,\rho)}(e_L)$ and $W\in\mathcal V_{(H,\sigma')}(e_H)$, there exists $V\in\VH$ such that $f\amu(V)=L\times V$ is $f\amu(f(U))$-thin, i.e. $L\times W$-thin, and for every $g=(l,h)\in G$ there exists $V_g\in\VH$ such that $[g,f\amu(V_g)]=[g,L\times V_g]\subseteq f\amu(f(U))=L\times W$. This is equivalent to say that $V$ is $W$-thin and $[h,V_g]\subseteq W$. So $f$ is A-open if and only if $1_H:(H,\sigma')\to(H,\sigma)$ is semitopological.

\smallskip
(ii) The homomorphism $f$ is strongly A-open if and only if for every $U=U'\times W\in\VG$, where $U'\in\mathcal V_{(L,\rho)}(e_L)$ and $W\in\mathcal V_{(H,\sigma')}(e_H)$, there exists $V\in\VH$ such that $f\amu(V)=L\times V$ is $U$-thin and for every $g=(l,h)\in G$ there exists $V_g\in\VH$ such that $[g,f\amu(V_g)]=[g,L\times V_g]\subseteq U=L\times W$. This is equivalent to say that $L$ is $U'$-thin, $V$ is $W$-thin and $[l,L]\subseteq U'$ and $[h,V_g]\subseteq W$. Since this happens for every $U'\in \mathcal V_{(L,\rho)}(e_L)$ and $l\in L$ we can conclude that $f$ is strongly A-open if and only if $(L,\rho)$ is SIN, for every $W\in\mathcal V_{(H,\sigma')}(e_H)$ there exists $V\in\VH$ such that $V$ is $W$-thin and $L'\subseteq N_\rho$, for every $W\in\mathcal V_{(H,\sigma')}(e_H)$ and for every $h\in H$ there exists $V_h\in\VH$ such that $[h,V_h]\subseteq W$. Since $L'\subseteq N_\rho$ yields that $(L,\rho)$ is SIN, $f$ is strongly A-open if and only if for every $W\in\mathcal V_{(H,\sigma')}$ there exists $V\in\VH$ such that $V$ is $W$-thin and for every $h\in H$ there exists $V_h\in\VH$ such that $[h,V_h]\subseteq W$, i.e. $L'\subseteq N_\tau$ and $1_H:(H,\sigma')\to(H,\sigma)$ is semitopological.

\smallskip
(a)$\Rightarrow$(b) Suppose that $f$ is A-open. We prove that the section $s:H\to G$ of $f$, defined by $s(h)=(e_L,h)$ for every $h\in H$, makes $f$ A$^*$-open. Let $U=U'\times W\in\VG$, where $U'\in\mathcal V_{(L,\rho)}(e_L)$ and $W\in\mathcal V_{(H,\sigma')}(e_H)$. By (i) there exists $V\in\VH$ such that $V$ is $W$-thin and so $s(V)=\{e_L\}\times V$ is $U$-thin. Let $g=(l,h)\in G$. By (i) there exists $V_h\in\VH$ such that $[h,V_h]\subseteq W$. Hence $[g,s(V_g)]=[(l,h),\{e_l\}\times V_h]\subseteq U$. This proves that $f$ is A$^*$-open.
\end{proof}

In Proposition \ref{compositions..} we see that item (i) holds without any restriction on the continuous surjective group homomorphism $f$.

\begin{proposition}\label{H-reflection}
Let $f:(G,\tau)\to(H,\sigma)$ be a continuous surjective homomorphism of topological groups. Let $q:(H,\sigma)\to(H/N_\sigma,\sigma_q)$ be the canonical projection and $f^*=q\circ f$ the \emph{Hausdorff reflection} of $f$.
\begin{itemize}
	\item[(a)]If $f$ is A-open then $f^*$ is A-open.
	\item[(b)]If $f$ is A$^*$-open then $f^*$ is A$^*$-open.
	\item[(c)]If $f$ is strongly A-open then $f^*$ is strongly A-open.		
\end{itemize}
\end{proposition}
\begin{proof}
(a) Let $U\in\VG$. There exists $V\in\VH$ such that $f\amu(V)$ is $f\amu(f(U))$-thin and for $g\in G$ there exists $V_g\in\VH$ such that $[g,f\amu(V_g)]\subseteq f\amu(f(U))$. There exist $V', V_g'\in \VH$ such that $V'V'\subseteq V$ and $V_g'V_g'\subseteq V_g$. Let $W=q(V')$ and $W_g=q(V_g')$. Then $$(f^*)\amu(W)=f\amu(q\amu(q(V')))=f\amu(V' N_\sigma)\subseteq f\amu(V)$$ and analogously $(f^*)\amu(W_g)\subseteq f\amu(V_g)$. Therefore $(f^*)\amu(W)$ is $f\amu(f(U))$-thin and $[g,(f^*)\amu(W_g)]\subseteq f\amu(f(U))$. To conclude note that $f\amu(f(U))\subseteq (f^*)\amu(f^*(U))$.

\smallskip
(b) Let $s$ be a section of $f$ that witnesses that $f$ is A$^*$-open and let $s_q$ be a section of $q$. Then $s^*=s\circ s_q$ is a section of $f^*$. We prove that $s^*$ makes $f^*$ A$^*$-open. Let $U\in\VG$. There exists $V\in\VH$ such that $s(V)$ is $U$-thin and for $g\in G$ there exists $V_g\in\VH$ such that $[g,s(V_g)]\subseteq U$. Since every section of $q$ is continuous at $e_{H/N_\sigma}$, there exist $W, W_g\in \mathcal V_{(H/N_\sigma,\sigma_q)}(e_{H/N_\sigma})$ such that $s_q(W)\subseteq V$ and $s_q(W_g)\subseteq V_g$. Then $s^*(W)$ is $U$-thin and $[g,s^*(W_g)]\subseteq U$. 

\smallskip
(c) Let $U\in\VG$. There exists $V\in\VH$ such that $f\amu(V)$ is $U$-thin and for $g\in G$ there exists $V_g\in\VH$ such that $[g,f\amu(V_g)]\subseteq U$. There exist $V', V_g'\in \VH$ such that $V'V'\subseteq V$ and $V_g'V_g'\subseteq V_g$. Let $W=q(V')$ and $W_g=q(V_g')$. Then $$(f^*)\amu(W)=f\amu(q\amu(q(V')))=f\amu(V' N_\sigma)\subseteq f\amu(V)$$ and analogously $(f^*)\amu(W_g)\subseteq f\amu(V_g)$. Therefore $(f^*)\amu(W)$ is $U$-thin and $[g,(f^*)\amu(W_g)]\subseteq U$.
\end{proof}

\section{Specific cases}\label{particular}

\subsection{Discrete case and indiscrete case}

In this section, for a continuous surjective homomorphism of topological groups $f:(G,\tau)\to(H,\sigma)$, we consider the particular cases when either $\tau=\delta_G$ or $\sigma=\iota_H$ and see when $f$ is semitopological in these situations.

First we consider the case when the topology on the domain is discrete.

\begin{remark}\label{taimanov}
In \cite{Tai} Ta\u\i manov introduced the topology $T_G$ on a group $G$: let $F\in[G]^{<\omega}$ (i.e. $F$ is a finite subset of $G$) and let $$c_G(F)=\bigcap_{x\in F}c_G(x)$$ be the centralizer of $F$ in $G$. Then $\mathcal C=\{c_G(F):F\in[G]^{<\omega}\}$ is a family of subgroups of $G$ closed under finite intersections.
Then $T_G$ is the topology such that $\mathcal C$ is a local base at $e_G$ of $T_G$. A prebase of $\mathcal V_{(G,T_G)}(e_G)$ is $\{c_G(x):x\in G\}$.

The topological group $(G,T_G)$ is Hausdorff if and only if $Z(G)=\{e_G\}$.
In case $G$ is a finitely generated group, $T_G=\delta_G$ whenever $Z(G)$ is trivial, that is when $T_G$ is Hausdorff.
If $G$ is an abelian group then $T_G=\iota_G$. Indeed $N_{T_G}=Z(G)$.
\end{remark}

The next proposition concerns the case when the topology on the domain is discrete.

\begin{proposition}\label{discrhom}
Let $f:(G,\delta_G)\to (H,\sigma)$ be a continuous surjective homomorphism of topological groups. Then:
\begin{itemize}
\item[(a)]$f$ is A-open if and only if $c_H(h)$ is $\sigma$-open for every $h\in H$ (i.e. $\sigma\geq T_H$);
\item[(b)]$f$ is A$^*$-open if and only if there exists a section $s$ of $f$ such that for every $g\in G$ there exists $V_g\in\VH$ such that $s(V_g)\subseteq c_G(g)$ (i.e. $s(\sigma)\geq T_G\restriction_{s(H)}$);
\item[(c)]$f$ is strongly A-open if and only if $\ker f\subseteq Z(G)$ and $f(c_G(g))$ is $\sigma$-open for every $g\in G$ (i.e. $\sigma\geq(T_G)_q$, identifying $G/\ker f$ with $H$).
\end{itemize}
\end{proposition}
\begin{proof}
Note that $M$ is $\{e_G\}$-thin for every $M\subseteq G$.

\smallskip
(a) In this case $f$ is A-open if and only if for every $g\in G$ there exists $V_g\in\mathcal V_{(H,\sigma)}(e_H)$ such that $[g,f^{-1}(V_g)]\subseteq\ker f$.

Suppose that $f$ is A-open. Since $\{e_G\}\in\mathcal V_{(G,\delta_G)}(e_G)$, for every $g\in G$ there exists $V\in\mathcal V_{(H,\sigma)}(e_H)$ such that $[g,f^{-1}(V_g)]\subseteq f^{-1}(f(e_G))=\ker f$. Hence for every $h\in H$ there exists $g\in G$ such that $f(g)=h$ and $[h,V_g]=\{e_H\}$. Thus $V_g\subseteq c_H(h)$ and $c_H(h)$ is $\sigma$-open.

Assume that $c_H(h)$ is $\sigma$-open for every $h\in H$. Let $g\in G$. There exists $V_g\in\mathcal V_{(H,\sigma)}(e_H)$ such that $[f(g),V_g]=\{e_H\}$. Consequently $[g,f^{-1}(V_g)]\subseteq f^{-1}([f(g),V_g])\subseteq\ker f$. Then $f$ is A-open.

\smallskip
(b) In this situation $f$ is A$^*$-open if and only if there exists a section $s$ of $f$ such that for every $g\in G$ there exists $V_g\in\VH$ such that $[g,s(V_g)]=\{e_G\}$, that is $s(V_g)\subseteq c_G(g)$.

\smallskip
(c) In this case $f$ is strongly A-open if and only if for every $g\in G$ there exists $V_g\in\mathcal V_{(H,\sigma)}(e_H)$ such that $[g, f\amu(V_g)]=\{e_G\}$, that is $f^{-1}(V_g)\subseteq c_G(g)$. 

Suppose that $f$ is strongly A-open. Let $g\in G$. There exists $V_g\in\mathcal V_{(H,\sigma)}(e_H)$ such that $f^{-1}(V_g)\subseteq c_G(g)$. Then $V_g\subseteq f(c_G(g))$ and so $f(c_G(g))$ is $\sigma$-open. Since $\ker f\subseteq f^{-1}(V_g)\subseteq c_G(g)$ for every $g\in G$, then $\ker f\subseteq Z(G)$.

Assume that $\ker f\subseteq Z(G)$ and $f(c_G(g))$ is $\sigma$-open for every $g\in G$. Let $g\in G$. There exists $V_g\in\mathcal V_{(H,\sigma)}(e_H)$ such that $V_g\subseteq f(c_G(g))$. Then $f^{-1}(V_g)\subseteq f^{-1}(f(c_G(g)))= c_G(g)\ker f$. Since $\ker f\subseteq Z(G)\subseteq c_G(g)$, we can conclude that $f^{-1}(V_g)\subseteq c_G(g)$. Then $f$ is strongly A-open.
\end{proof}

As a corollary of this proposition we obtain the next result, which is \cite[Corollary 5]{Ar}. We write it in terms of the centralizer and the topology of Ta\u\i manov.

\begin{corollary}\label{discriso}
A continuous isomorphism of topological groups $f:(G,\delta_G)\to (H,\sigma)$ is semitopological if and only if $c_H(h)$ is $\sigma$-open for every $h\in H$ (i.e. $\sigma\geq T_H$).
\end{corollary}

The following are other consequences of Proposition \ref{discrhom}.

\begin{remark}
Let $f:(G,\delta_G)\to (H,\sigma)$ be a continuous surjective homomorphism of topological groups. 
\begin{itemize}
	\item[(a)]If $Z(H)$ is $\sigma$-open, then $f$ is A-open.
	\item[(b)]If $f$ is A$^*$-open, then there exists a section $s$ of $f$ such that $s(N_\sigma)\subseteq Z(G)$.
	\item[(c)]If $f(Z(G))$ is $\sigma$-open and $\ker f\subseteq Z(G)$, then $f$ is strongly A-open.
\end{itemize}
\end{remark}

\begin{corollary}\label{diab}
Let $G,H$ be groups and $f:(G,\delta_G)\to (H,\iota_H)$ a continuous surjective homomorphism. Then the following are equivalent:
\begin{itemize}
\item[(a)]$f:(G,\delta_G)\to(H,\iota_H)$ is A-open;
\item[(b)]$f:(G,\delta_G)\to(H,\iota_H)$ is semitopological;
\item[(c)]$H$ is abelian.
\end{itemize}
\end{corollary}
\begin{proof}
(a)$\Rightarrow$(c) by Proposition \ref{discrhom}(a), (c)$\Rightarrow$(b) by Corollary \ref{graph} and (b)$\Rightarrow$(a) follows from Theorem \ref{semitophom}.
\end{proof}

This corollary yields the following example, which shows that if $f:(G,\tau)\to (H,\sigma)$ is a continuous surjective homomorphism of topological groups such that $G\cong\ker f\times H$ (in particular there exists a section $s$ of $f$ which is a quasihomomorphism), $f$ need not be A-open.

\begin{example}\label{c,non-c*}
Let $H$ be a non-abelian group and $G=H\times H$. Then the canonical projection $f=p_2:(G,\delta_G)\to (H,\iota_H)$ is a continuous surjective homomorphism which admits a section $s$ that is a homomorphism. Since $H$ is non-abelian $f$ is not A-open by Corollary \ref{diab}. In particular $f$ is not A$^*$-open.
\end{example}

This example shows also that a section which is a quasihomomorphism is not always a section which makes $f$ strongly A$^*$-open.

\begin{proposition}\label{allsemitop=ab}
For a group $H$ the following are equivalent:
\begin{itemize}
\item[(a)]$H$ is abelian;
\item[(b)]every continuous surjective homomorphism $f:(G,\tau)\to (H,\sigma)$ is semitopological, for every group $G$ and for every group topologies $\tau,\sigma$;
\item[(c)]every continuous isomorphism $f:(G,\tau)\to (H,\sigma)$ is semitopological, for every group $G$ and for every group topologies $\tau,\sigma$;
\item[(d)]$1_H:(H,\delta_H)\to (H,\iota_H)$ is semitopological.
\end{itemize}
\end{proposition}
\begin{proof}
Corollary \ref{graph} yields (a)$\Rightarrow$(b)$\Rightarrow$(c)$\Rightarrow$(d). Moreover (d)$\Rightarrow$(a) by Corollary \ref{diab}.
\end{proof}

Proposition \ref{discrhom} has the following corollary.

\begin{corollary}
If $G$ is a group, $(H,\sigma)$ a connected topological group and $f:(G,\delta_G)\to(H,\sigma)$ a continuous surjective homomorphism, then the following are equivalent:
\begin{itemize}
	\item[(a)] $f$ is semitopological;
	\item[(b)] $f$ is A-open;
	\item[(c)] $H$ is abelian.
\end{itemize}
\end{corollary}
\begin{proof}
(a)$\Rightarrow$(b) by Theorem \ref{semitophom} and (c)$\Rightarrow$(a) by Corollary \ref{graph}.

\smallskip
(b)$\Rightarrow$(c) Suppose that $(H,\sigma)$ is Hausdorff. By Proposition \ref{discrhom}(a) $c_H(h)$ is $\sigma$-open for every $h\in H$. Since $(H,\sigma)$ is Hausdorff, then $c_H(h)$ is also $\sigma$-closed. Being $(H,\sigma)$ connected, $c_H(h)=H$ for every $h\in H$. Hence $H$ is abelian.

Consider the general case and let $q:(H,\sigma)\to(H/N_\sigma,\sigma_q)$ be the canonical projection and $f^*=q\circ f$. Since $f$ is A-open, then $f^*$ is A-open by Proposition \ref{H-reflection}(a). Since $(H/N_\sigma,\sigma_q)$ is connected and Hausdorff, apply the previous case to conclude that $H$ is abelian.
\end{proof}

In \cite{Ar} the case when the topology on the codomain is indiscrete was not considered and this is done in Proposition \ref{indhom} for surjective homomorphisms and in its Corollary \ref{indiso} for isomorphisms.

\begin{proposition}\label{indhom}
Let $f:(G,\tau)\to (H,\iota_H)$ be a continuous surjective homomorphism of topological groups. Then:
\begin{itemize}
\item[(a)]$f$ is A-open if and only if $G'\cdot\ker f\subseteq\overline{\ker f}$;
\item[(b)]$f$ is A$^*$-open if and only if there exists a section $s$ of $f$ such that $s(H)$ is $U$-thin for every $U\in\VG$ and $[G,s(H)]\subseteq N_\tau$;
\item[(c)]$f$ is strongly A-open if and only if $G'\subseteq N_\tau$;
\item[(d)]$f$ is semitopological if and only if there exists a topological group $(\widetilde G,\widetilde \tau)$ of which $(G,\tau)$ is a topological normal subgroup and such that $\widetilde G=NG$ and $N\cap G=\ker f$ for some dense normal subgroup $N$ of $\widetilde G$ (the last two conditions ensure that the quotient topology on $\widetilde G/N$ is indiscrete and so $(\widetilde G/N,\widetilde\tau_q)$ is topologically isomorphic to $(H,\iota_H)$).
\end{itemize}
\end{proposition}
\begin{proof}
Let $q:G\to G/\ker f$ be the canonical projection.

\smallskip
(a) We prove first that $G'\ker f\subseteq \overline{\ker f}$ if and only if $(G/\ker f)'\subseteq N_{\tau_q}$ and then that $f$ is A-open if and only if $(G/\ker f)'\subseteq N_{\tau_q}$.

Since $q(\overline{\ker f})=\overline{\ker f}/\ker f$ is indiscrete, then $q(\overline{\ker f})\subseteq N_{\tau_q}$. But $q(\overline{\ker f})$ is closed, because $\overline{\ker f}\supseteq \ker f=\ker q$, and so $q(\overline{\ker f})=N_{\tau_q}$. Since $q(G'\ker f)=q(G')=(G/\ker f)'$, so $(G/\ker f)'\subseteq N_{\tau_q}=q(\overline{\ker f})$ if and only if $q(G'\ker f)\subseteq q(\overline{\ker f})$. Since both $G'\ker f$ and $\overline{\ker f}$ contain $\ker f$, so $q(G'\ker f)\subseteq q(\overline{\ker f})$ is equivalent to $G'\ker f\subseteq \overline{\ker f}$.

Assume that $f$ is A-open and let $U\in\mathcal V_{(G,\tau)}(e_G)$. Then $G$ is $f^{-1}(f(U))$-thin.
Since $f$ is A-open, $[g,G]\subseteq f^{-1}(f(U))$ for every $g\in G$ and for every $U\in\mathcal V_{(G,\tau)}(e_G)$. Then $[q(g),G/\ker f]\subseteq q(U)$ for every $U\in\mathcal V_{(G,\tau)}(e_G)$. So $(G/\ker f,)'\subseteq N_{\tau_q}$.

Suppose $(G/\ker f)'\subseteq N_{\tau_q}$. Let $U\in\mathcal V_{(G,\tau)}(e_G)$. From the assumption it follows that $(G/\ker f,\tau_q)$ is SIN, so $G/\ker f$ is $q(U)$-thin. By Lemma \ref{thin}(f) $q\amu(G/\ker f)=G$ is $q\amu(q(U))=f^{-1}(f(U))$-thin for every $U\in\mathcal V_{(G,\tau)}(e_G)$.
Let $U\in\mathcal V_{(G,\tau)}(e_G)$ and $g\in G$. Then $q([g,x])=[q(g),q(x)]\in(G/\ker f)'\subseteq q(U)$, since $G/\ker f\subseteq N_{\tau_q}$ and for every $x\in G$. Therefore $[g,x]\in U\ker f=f^{-1}(f(U))$ and this yields $[g,G]\subseteq f^{-1}(f(U))$.

\smallskip
(b) In this case $f$ is A$^*$-open if and only if there exists a section $s$ of $f$ such that for every $U\in\VG$ $s(H)$ is $U$-thin and for each $g\in G$ $[g,s(H)]\subseteq U$. This is equivalent to $[g,s(H)]\subseteq N_\tau$ for every $g\in G$, that is $[G,s(H)]\subseteq N_\tau$.

\smallskip
(c) Suppose that $f$ is strongly A-open. Moreover for every $U\in\mathcal V_{(G,\tau)}(e_G)$ and for every $g\in G$ we have $[g,G]\subseteq U$ and so $G'\subseteq N_\tau$.

Assume that $G'\subseteq N_\tau$. Let $U\in\mathcal V_{(G,\tau)}(e_G)$. Since $G'\subseteq N_\tau$, $G$ is SIN, so $G$ is $U$-thin and $[g,G]\subseteq N_\tau\subseteq U$ for every $g\in G$.

\smallskip
(d) If $f$ is semitopological there exists an A-extension $(\widetilde G,\widetilde \tau)$ with $\widetilde f:(\widetilde G,\widetilde\tau)\to(H,\sigma)$ open continuous surjective homomorphism. Then $N=\ker \widetilde f$ has the required properties by Proposition \ref{firstprop}.

Suppose that there exists a dense normal subgroup $N$ of $(\widetilde G,\widetilde\tau)$ such that $\widetilde G=G N$ and $N\cap G=\ker f$. Let $\widetilde f:\widetilde G\to H$ be defined by $\widetilde f(g n)=f(g)$ for every $g\in G,n\in N$. Note that $\widetilde G/N$ is algebraically isomorphic to $G/\ker f$ and so to $H$. Being $N$ dense in $(\widetilde G,\widetilde \tau)$, $\widetilde\tau_q$ on $\widetilde G/N$ is the indiscrete topology and so $(\widetilde G/N,\iota_{\widetilde G/N})$ is topologically isomorphic to $(H,\iota_H)$. Hence $\widetilde f$ is open and $(\widetilde G,\widetilde\tau)$ with $\widetilde f$ is an A-extension of $f$, which in particular is semitopological.
\end{proof}

As noted in the foregoing proof the condition $G'\subseteq N_\tau$ yields that $(G,\tau)$ is SIN.

\medskip
In the next corollary of Proposition \ref{indhom} we consider the case when the topology on the domain is Hausdorff and the topology on the codomain is indiscrete.

\begin{corollary}\label{T2}
Let $f:(G,\tau)\to(H,\iota_H)$ be a homomorphism of topological groups, where $\tau$ is Hausdorff.
\begin{itemize}
	\item[(a)]Then $f$ is strongly A-open if and only if $G$ is abelian.
	\item[(b)]If $f$ is A$^*$-open, then $H$ is abelian and so $f$ is semitopological.
	\item[(c)]If $\ker f$ is closed in $(G,\tau)$, then the following are equivalent:
\begin{itemize}
	\item[(i)]$f$ is A-open;
	\item[(ii)]$f$ is semitopological;
	\item[(iii)]$H$ is abelian.
\end{itemize}
\end{itemize}
\end{corollary}

The next corollary is the counterpart of Corollary \ref{discriso} in case the second topology is indiscrete. It follows immediately from Proposition \ref{indhom}.

\begin{corollary}\label{indiso}
A continuous isomorphism of topological groups $f:(G,\tau)\to (H,\iota_H)$ is semitopological if and only if $G'\subseteq N_{\tau}$.
\end{corollary}

The following corollary is an obvious consequence of Corollary \ref{indiso} (or of Corollary \ref{T2}).

\begin{corollary}\label{hausdorff}
Let $f:(G,\tau)\to(H,\iota_H)$ be an isomorphism of topological groups. If $\tau$ is Hausdorff, then the following conditions are equivalent:
\begin{itemize}
\item[(a)]$f:(G,\tau)\to(H,\iota_H)$ is semitopological;
\item[(b)]$G$ is abelian;
\item[(c)]$H$ is abelian.
\end{itemize}
\end{corollary}
\begin{proof}
(a)$\Rightarrow$(b) By Corollary \ref{indiso} $G'\subseteq N_\tau$. Since $\tau$ is Hausdorff, $G'$ is trivial and so $G$ is abelian.

\smallskip
(b)$\Rightarrow$(c) is obvious and (c)$\Rightarrow$(a) follows from Corollary \ref{graph}.
\end{proof}

The next example shows that in general neither semitopological nor strongly A-open imply the existence of a section which is a quasihomomorphism. Moreover this is an example of a continuous surjective homomorphism of abelian groups which has no section that is a quasihomomorphism.

\begin{example}\label{strAop-nonquasihomo}
Let $G,H$ be abelian groups and $f:(G,\delta_G)\to (H,\iota_H)$ a continuous surjective homomorphism such that $\ker f$ is not a direct summand of $G$; in particular $G\not\cong \ker f\times H$. Then $f$ is strongly A-open by Corollary \ref{hausdorff} and in particular it is semitopological by Corollary \ref{diab}. Let $s$ be a section of $f$. By Remark \ref{quasihom=hom} if $s$ is a quasihomomorphism then $s$ is a homomorphism and $s$ cannot be a homomorphism because $G\not\cong\ker f\times H$.
\end{example}

\subsection{Almost discrete topologies}

To find some examples and counterexamples we consider topologies which are very close to be discrete.

\begin{deff}
A topological group $(G,\tau)$ is \emph{almost discrete} if $N_\tau$ is open in $(G,\tau)$.
\end{deff}

First examples of almost discrete group topologies on a group $G$ are $\delta_G$ ($N_{\delta_G}=\{e_G\}$) and $\iota_G$ ($N_{\iota_G}=G$). Quotients of almost dicrete groups are almost discrete.
Every finite topological group is almost discrete and every almost discrete group is SIN.

\begin{example}
If $G$ is a finitely generated, then $(G,T_G)$ is almost discrete, where $T_G$ is the topology of centralizers introduced by Ta\u\i manov --- see Remark \ref{taimanov}.
\end{example}

\begin{lemma}
Let $(G,\tau),(H,\sigma)$ be almost discrete groups and $f:(G,\tau)\to (H,\sigma)$ a continuous surjective homomorphism. Then:
\begin{itemize}
\item[(a)]$f$ is A-open if and only if $[G,f^{-1}(N_\sigma)]\subseteq N_\tau\ker f$;
\item[(b)]$f$ is A$^*$-open if and only if there exists a section $s$ of $f$ such that $[G,s(N_\sigma)]\subseteq N_\tau$;
\item[(c)]$f$ is strongly A-open if and only if $[G,f^{-1}(N_\sigma)]\subseteq N_\tau$.
\end{itemize}
\end{lemma}
\begin{proof}
Since $f^{-1}(N_\sigma)$ is $N_\tau$-thin, then $f$ is A-open if and only if for every $U\supseteq N_\tau$ and for every $g\in G$ there exists $V_g\supseteq N_\sigma$ such that $[g,f^{-1}(V_g)]\subseteq U\ker f$; moreover $f$ is A$^*$-open if and only if there exists a section $s$ of $f$ such that $s(e_H)=e_G$ and for every $U\supseteq N_\tau$ and for every $g\in G$ there exists $V_g\supseteq N_\sigma$ such that $[g,s(V_g)]\subseteq U$; finally $f$ is strongly A-open if and only if for every $U\supseteq N_\tau$ and for every $g\in G$ there exists $V_g\supseteq N_\sigma$ such that $[g,f^{-1}(V_g)]\subseteq U$.

\smallskip
(a) Suppose that $f$ is A-open. Then for every $g\in G$ there exists $V_g\supseteq N_\sigma$ such that $[g,f^{-1}(V_g)]\subseteq N_\tau\ker f$. Hence $[g,f^{-1}(N_\sigma)]\subseteq N_\tau\ker f$ for every $g\in G$ and so $[G,f^{-1}(N_\sigma)]\subseteq N_\tau\ker f$.

Assume that $[G,f^{-1}(N_\sigma)]\subseteq N_\tau\ker f$ and let $U\supseteq N_\tau$. For every $g\in G$ we have $$[g,f^{-1}(N_\sigma)]\subseteq [G,f^{-1}(N_\sigma)]\subseteq N_\tau\ker f\subseteq f^{-1}(f(U)).$$ Thus $f$ is A-open.

\smallskip
(b) Suppose that $f$ is A$^*$-open. Then there exists a section $s$ of $f$ such that for every $g\in G$ there exists $V_g\supseteq N_\sigma$ such that $[g,s(V_g)]\subseteq N_\tau$. Hence $[g,s(N_\sigma)]\subseteq N_\tau$ for every $g\in G$ and so $[G,s(N_\sigma)]\subseteq N_\tau$.

Assume that there exists a section $s$ of $f$ such that $[G,s(N_\sigma)]\subseteq N_\tau$ and let $U\supseteq N_\tau$. For every $g\in G$ we have $[g,s(N_\sigma)]\subseteq [G,s(N_\sigma)]\subseteq N_\tau\subseteq U$. Thus $f$ is A$^*$-open.

\smallskip
(c) Suppose that $f$ is strongly A-open. Then for every $g\in G$ there exists $V_g\supseteq N_\sigma$ such that $[g,f^{-1}(V_g)]\subseteq N_\tau$. Hence $[g,f^{-1}(N_\sigma)]\subseteq N_\tau$ for every $g\in G$ and so $[G,f^{-1}(N_\sigma)]\subseteq N_\tau$.

Assume that $[G,f^{-1}(N_\sigma)]\subseteq N_\tau$ and let $U\supseteq N_\tau$. For every $g\in G$ we have $[g,f^{-1}(N_\sigma)]\subseteq [G,f^{-1}(N_\sigma)]\subseteq N_\tau\subseteq U$. Thus $f$ is strongly A-open.
\end{proof}

The next result is a consequence of Propositions \ref{discrhom} and \ref{indhom}.

\begin{corollary}
Let $(G,\tau)$, $(H,\sigma)$ be almost discrete topological groups and $f:(G,\tau)\to(H,\sigma)$ a continuous surjective homomorphism.
\begin{itemize}
\item[(1)]If $\tau=\delta_G$, then:
\begin{itemize}
\item[(a)]$f:(G,\delta_G)\to(H,\sigma)$ is A-open if and only if $N_\sigma\subseteq Z(H)$;
\item[(b)]$f:(G,\delta_G)\to(H,\sigma)$ is A$^*$-open if and only if there exists a section $s$ of $f$ such that $s(N_\sigma)\subseteq Z(G)$;
\item[(c)]$f:(G,\delta_G)\to(H,\sigma)$ is strongly A-open if and only if $\ker f\subseteq Z(G)$ and $f^{-1}(N_\sigma)\subseteq Z(G)$.
\end{itemize}
\item[(2)]If $\sigma=\iota_H$, then:
\begin{itemize}
\item[(a)]$f:(G,\tau)\to(H,\iota_H)$ is A-open if and only if $G'\subseteq N_\tau\ker f$;
\item[(b)]$f:(G,\tau)\to(H,\iota_H)$ is A$^*$-open if and only if there exists a section $s$ of $f$ such that $[G,s(H)]\subseteq N_\tau$;
\item[(c)]$f:(G,\tau)\to(H,\iota_H)$ is strongly A-open if and only if $G'\subseteq N_\tau$.
\end{itemize}
\item[(3)]If $f$ is an isomorphism, then $f:(G,\tau)\to(H,\sigma)$ is semitopological if and only if $[G,f^{-1}(N_\sigma)]\subseteq N_\tau$.
\end{itemize}
\end{corollary}

Let $G$ be a non-trivial group. The \emph{upper central series} of $G$ is defined by $Z_0(G)=\{e_G\}$, $Z_1(G)=Z(G)$ and $Z_n(G)$ is such that $Z_n(G)/Z_{n-1}(G)=Z(G/Z_{n-1}(G))$ for every $n\in\mathbb N$, $n\geq 2$. A group $G$ is \emph{nilpotent} if and only if $Z_m(G)=G$ for some $m\in\mathbb N_+$.

\begin{example}\label{almdiscrex}
Let $G$ be a group and $\tau$ an almost discrete group topology on $G$. Consider $$(G,\delta_G)\xrightarrow{1_G}(G,\tau)\xrightarrow{1_G}(G,\iota_G).$$
Then:
\begin{itemize}
\item[(a)]$1_G:(G,\delta_G)\to(G,\tau)$ is semitopological if and only if $N_\tau\subseteq Z(G)$;
\item[(b)]$1_G:(G,\tau)\to(G,\iota_G)$ is semitopological if and only if $G'\subseteq N_\tau$.
\end{itemize}
\end{example}

In particular this example shows that in general the composition of semitopological isomorphisms could be not semitopological: if $G$ is non-abelian and $G'\subseteq N_\tau\subseteq Z(G)$, where $N_\tau\neq\{e_G\}$, then $1_G:(G,\delta_G)\to(G,\tau)$ and $1_G:(G,\tau)\to(G,\iota_G)$ are semitopological, but $1_G:(G,\delta_G)\to(G,\iota_G)$ is not semitopological.

If $G$ is nilpotent of class $2$ and $N_\tau=Z(G)\neq\{e_G\}$, then $G'\subseteq N_\tau=Z(G)$, because $G/Z(G)$ is abelian, and we have exactly \cite[Example 12]{Ar}.

\medskip
As a corollary of the previous example we have the following one, which shows that strongly A$^*$-open does not imply open.

\begin{example}\label{non-open_*-section}
Let $G$ be a non-perfect group, that is $G'\neq G$, and let $\tau$ be the almost discrete group topology on $G$ such that $N_\tau=G'$. Then $1_G:(G,\tau)\to(G,\iota_G)$ is not open and it is semitopological by the previous example. In this case it is equivalent to say that $1_G:(G,\tau)\to(G,\iota_G)$ is strongly A$^*$-open (see Remark \ref{iso->all_equivalent}).
\end{example}

\begin{claim}
Let $(G,\tau),(H,\sigma)$ be topological groups, $f:(G,\tau)\to(H,\sigma)$ a continuous surjective homomorphism and $s$ a section of $f$. Let $q:G\to G/N_\tau$ be the canonical projection. If $\sigma$ is almost discrete, then $s$ is a quasihomomorphism if and only if $q\circ s\restriction_{N_\sigma}:N_\sigma\to G/N_\tau$ is a homomorphism.

In case $\tau$ is Hausdorff, $s$ is a quasihomomorphism if and only if $s\restriction_{N_\sigma}$ is a homomorphism.
\end{claim}
\begin{proof}
Suppose that $s$ is a quasihomomorphism. Then $s':(H\times H,\sigma\times\sigma)\to (G,\tau)$, defined by $s'(x,y)= s(x)s(y)s(x y)\amu$ for every $x,y\in H$, is continuous at $(e_H,e_H)$. This means that $s(x)s(y)s(xy)\amu\in \bigcap_{U\in\VG}U=N_\tau$ for every $x,y\in N_\sigma$. Hence $q(s(x)s(y)s(x y)\amu)=e_{G/N_\tau}$ and so $q(s(x))q(s(y))=q(s(xy))$ for every $x,y\in N_\sigma$.

To prove the converse implication assume that $q\circ s\restriction_{N_\sigma}$ is a homomorphism. Let $x,y \in N_\sigma$. Then $$q(s(x y))=q(s(x))q(s(y))=q(s(x)s(y))$$ and consequently $s'(x,y)=s(x)s(y)s(x y)\amu\in N_\tau$ for every $x,y\in H$. Moreover $$q(s(y)s(x)s(y)\amu s(y x y\amu)\amu)=q(s(y))q(s(x))q(s(y))\amu(q(s(y))q(s(x))q(s(y))\amu)\amu=e_G$$ yields $s_y(x)=s(y)s(x)s(y)\amu s(y x y\amu)\amu\in N_\tau$ for every $x,y\in H$. Since $N_\tau\subseteq U$ for every $U\in\VG$, these two properties prove that $s$ is a quasihomomorphism. 
\end{proof}

The next example shows that the property A-open could be strictly weaker than A$^*$-open.

\begin{example}\label{Aop-<A*op}
Let $G$ be a nilpotent group of class $3$ and $f=q:G\to G/Z(G)=H$ the canonical projection. Endow $G$ with $\delta_G$ and $H$ with the almost discrete topology $\sigma$ such that $N_\sigma=Z_2(G)/Z(G)$. Then $f$ is A-open, but not A$^*$-open.

Let $U\in\VG$. Then $f\amu(V)$ is $U$-thin for every $V\in\VH$. For $U'=\{e_G\}$ and $V=N_\sigma$, $$[G,f\amu(V)]=[G,Z_2(G)]\subseteq f\amu(f(U'))=Z(G).$$ Then for every $U\in\VG$ and for every $g\in G$, $[g,f\amu(V)]\subseteq f\amu(f(U))$. This proves that $f$ is A-open.
If there exists a section $s$ of $f$ such that $s(e_H)=e_G$ and $f$ is A$^*$-open, taking $U=\{e_G\}\in\VG$, then $[g,s(N_\sigma)]=\{e_G\}$ for every $g\in G$. This implies that $s(N_\sigma)\subseteq Z(G)$, but it is not possible because $s(N_\sigma)\cap Z(G)=\{e_G\}$ and $s(N_\sigma)\neq\{e_G\}$.
\end{example}

In the following diagram we collect the relations among the main properties we have considered.

\begin{equation}\label{small_diagram}
\xymatrix{
& \text{strongly A$^*$-open} \ar@{=>}[ldd] \ar@{=>}[rdd] \ar[rd]^{no} \ar[ld]_{no} & \\
\text{open} \ar@{=>}[d] \ar[rr]^{no} \ar@{-->}@/^1pc/[ru]^{?}  \ar@/_0.8pc/@{-->}[rrd]^? &  & \text{strongly A-open} \ar@{=>}[d] \ar@/_1pc/[lu]_{no} \ar@/_1pc/[ll]_{no} \ar@/_0.7pc/@{-->}[lld]_? \\
\text{semitopological} \ar@{=>}[dr] \ar@/^0.7pc/[u]^{no} \ar@{-->}@/_0.7pc/[rr]_? \ar[rru]^{no} & & \text{A$^*$-open} \ar@{=>}[dl] \ar@/_0.7pc/[u]_{no} \ar@{-->}[ll]_? \\
& \text{A-open} \ar@{-->}@/^1pc/[lu]^? \ar@/_1pc/[ru]_{no} &
}
\end{equation}

\section{Stability}\label{stability}

\subsection{Stability under pullback}

The next theorem generalizes Theorem \ref{subgroupst}(a). Items (a), (b) and (c) below say equivalently that the classes of all A-open, A$^*$-open and strongly A$^*$-open continuous surjective group homomorphisms are closed under pullback along topological embeddings. Item (d) is stronger and implies that the class of all strongly A-open continuous surjective group homomorphisms is closed under pullback along topological embeddings.

\begin{theorem}\label{subgroups_stability}
Let $(G,\tau),(H,\sigma)$ be topological groups, $N$ a subgroup of $H$, $M$ a subgroup of $G$ and $f:(G,\tau)\to(H,\sigma)$ a continuous surjective homomorphism.
\begin{itemize}
\item[(a)]If $f:(G,\tau)\to(H,\sigma)$ is A-open, then $f\restriction_{f^{-1}(N)}:(f^{-1}(N),\tau\restriction_{f^{-1}(N)})\to(N,\sigma\restriction_N)$ is A-open.
\item[(b)]If $f:(G,\tau)\to(H,\sigma)$ is A$^*$-open, then $f\restriction_{f^{-1}(N)}:(f^{-1}(N),\tau\restriction_{f^{-1}(N)})\to(N,\sigma\restriction_N)$ is A$^*$-open.
\item[(c)]If $f:(G,\tau)\to(H,\sigma)$ is strongly A$^*$-open, then $f\restriction_{f^{-1}(N)}:(f^{-1}(N),\tau\restriction_{f^{-1}(N)})\to(N,\sigma\restriction_N)$ is strongly A$^*$-open.
\item[(d)]If $f:(G,\tau)\to(H,\sigma)$ is strongly A-open, then $f\restriction_M:(M,\tau\restriction_M)\to(f(M),\sigma\restriction_{f(M)})$ is strongly A-open.
\end{itemize}
\end{theorem}
\begin{proof}
(a) Note that if $S$ is a subset of $G$, then $$f^{-1}(f(S))\cap f^{-1}(N)=f^{-1}(f(S)\cap f^{-1}(f(S\cap f^{-1}(N))).$$ To prove this let $y\in f^{-1}(f(S))\cap f^{-1}(N)$. Then $f(y)\in f(S\cap f^{-1}(N))\subseteq f(S)\cap N$. So $f^{-1}(f(S))\cap f^{-1}(N)\subseteq f^{-1}(f(S)\cap f^{-1}(f(S\cap f^{-1}(N)))$. Now let $y\in f^{-1}(f(S)\cap f^{-1}(f(S\cap f^{-1}(N)))$. Thus $f(y)\in f(S)\cap N\subseteq f(S\cap f^{-1}(N))$ and hence $f^{-1}(f(S))\cap f^{-1}(N)\supseteq f^{-1}(f(S)\cap f^{-1}(f(S\cap f^{-1}(N)))$.

Let $U\in\mathcal V_{(G,\tau)}(e_G)$. There exists $V\in\VH$ such that $f^{-1}(V)$ is $f^{-1}(f(U))$-thin. So there exists $U'\in\VG$ such that $x U' x^{-1}\subseteq f^{-1}(f(U))$ for every $x\in f^{-1}(V)$. Then $x(U'\cap f^{-1}(N))x^{-1}\subseteq f^{-1}(f(U))\cap f^{-1}(N)$ for every $x\in f^{-1}(V)\cap f^{-1}(N)$. But $$f^{-1}(f(U))\cap f^{-1}(N)=f^{-1}(f(U)\cap f^{-1}(f(U\cap f^{-1}(N))).$$  Moreover $f^{-1}(V\cap N)\subseteq f^{-1}(V)\cap f^{-1}(N)$. Therefore we have proved that $x(U'\cap f^{-1}(N))x^{-1}\subseteq f^{-1}(f(U\cap f^{-1}(N)))$ for every $x\in f^{-1}(V\cap N)$, that is $f^{-1}(V\cap N)$ is $f^{-1}(f(U\cap f^{-1}(N)))$-thin.

Let $g\in f^{-1}(N)$. There exists $V_g\in\VH$ such that $[g,f^{-1}(V_g)]\subseteq f	\amu(f(U))$. Then $[g,f^{-1}(V_g\cap N)]\subseteq f\amu(f(U))\cap f\amu(N)=f\amu(f(U\cap f\amu(N))))$.

\smallskip
(b) Let $s:H\to G$ be a section of $f$ that witnesses that $f$ is A$^*$-open.  Since $s(N)=s(H)\cap f\amu(N)$, $s\restriction_N: N\to f\amu(N)$ is a section of $f\restriction_{f\amu(N)}$. Let $U\in \VG$. There exists $V\in\VH$ such that $s(V)$ is $U$-thin, that is there exists $U'\in\VG$ such that $x U'x\amu\subseteq U$ for every $x\in s(V)$. Hence $x(U'\cap f\amu(N))x\amu\subseteq U\cap f\amu(N)$ for every $x\in s\restriction_N(V\cap N)$, because $s\restriction_N(V\cap N)\subseteq s(V)\cap f\amu(N)$.
Let $g\in f\amu(N)$. There exists $V_g\in\VH$ such that $[g,s(V_g)]\subseteq U$. Then $[g,s\restriction_N(V_g\cap N)]\subseteq U\cap f\amu (N)$.

\smallskip
(c) Let $s:H\to G$ be a section of $f$ that witnesses that $f$ is strongly A$^*$-open. By the proof of (b) $s\restriction_N$ is a section of $f\restriction_{f\amu(N)}$ that witnesses that $f\restriction_{f\amu(N)}$ is A$^*$-open. Now we have that $s$ is a also a quasihomomorphism and we need only to prove that $s\restriction_N$ is a quasihomomorphism. This is true because $s'(N,N)\cup s_y(N)\subseteq f\amu(N)$ for every $y\in N$.

\smallskip
(d) Observe that $M\supseteq(f\restriction_M)\amu(S\cap f(M))$ whenever $S\subseteq H$.
Let $U\in\VG$. There exists $V\in\VH$ such that $f\amu(V)$ is $U$-thin, that is there exists $U'\in\VG$ such that $x U' x\amu\subseteq U$ for every $x\in f\amu(V)$.  Then $x(U'\cap M)x\amu\subseteq U\cap M$ for every $x\in (f\restriction_M)\amu(V\cap f(M))$. Let $g\in M$. There exists $V_g\in\VH$ such that $[g,f\amu(V_g)]\subseteq U$. Hence $[g,(f\restriction_M)\amu(V_g\cap f(M))]\subseteq U\cap M$.
\end{proof}

As noted for the previous theorem, in categorical terms Theorem \ref{subgroupst}(a) says that $\mathcal S_i$ is closed under pullback with respect to topological embeddings. From now on we consider only one group as support for two group topologies one stronger than the other and as the continuous isomorphisms we take the identity map. In this setting the topological embedding is the embedding of a topological subgroup. As shown in the introduction this causes no loss of generality. In the following diagram $G$ is a group and $\tau\geq\sigma$ group topologies on $G$, while $N$ is a subgroup of $G$ and $1_N:(N,\tau\restriction_N)\to(N,\sigma\restriction_N)$ is the pullback of $1_G:(G,\tau)\to(G,\sigma)$ with respect to the topological embedding $(N,\sigma\restriction_N)\to(G,\sigma)$.
\begin{equation*}
\xymatrix{
(G,\tau) \ar[r] & (G,\sigma)\\
(N,\tau\restriction_N) \ar[r]\ar@{^{(}->}[u] & (N,\sigma\restriction_N)\ar@{^{(}->}[u]
}
\end{equation*}
We consider in Theorem \ref{pullback} the more general stability of $S_i$ under pullback along continuous injective homomorphisms.

\begin{lemma}\label{pullback_lemma}
Let $G$ be a group and $\tau\geq \sigma \leq\sigma'$ group topologies on $G$. If $1_G:(G,\tau)\to (G,\sigma)$ is semitopological, then $1_G:(G,\tau \vee\sigma')\to (G,\sigma')$ is semitopological as well, where $(G,\tau \vee\sigma')$ is the pullback with respect to $1_G:(G,\sigma')\to(G,\sigma)$.
\end{lemma}
\begin{proof}
The situation of the topologies is the following:
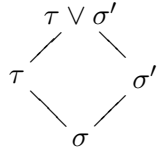
\begin{equation*}
\xymatrix@!0{
& \tau\vee\sigma'\ar@{-}[dl]\ar@{-}[dr] & \\
\tau \ar@{-}[dr] &  & \sigma' \ar@{-}[dl] \\
& \sigma &
}
\end{equation*}
A base of $\mathcal V_{(G,\tau\vee\sigma')}(e_G)$ is $\{U\cap W:U\in\VG, V\in\mathcal V_{(G,\sigma')} (e_G)\}$.

By Remark \ref{iso->all_equivalent} in case of isomorphisms semitopological is equivalent to strongly A-open.
Let $U\cap W\in \mathcal V_{(G,\tau\vee\sigma')}(e_G)$. There exists $V\in\mathcal V_{(G,\sigma)}(e_G)$ such that $V$ is $U$-thin, i.e. there exists $U'\in\VG$ such that $x U' x\amu\subseteq U$ for every $x\in V$. Since $\sigma'\geq\sigma$, there exists $W\in\mathcal V_{(G,\sigma')}(e_G)$ such that $W'W'W'\subseteq W$ and $W'\subseteq V$. Then $x(U'\cap W')x\amu\subseteq U\cap W$ for every $x\in W'$, that is $W'$ is $U\cap W$-thin. Let $g\in G$. There exists $V_g\in\mathcal V_{(G,\sigma)}(e_G)$ such that $[g,V_g]\subseteq U$. Since $\sigma'\geq\sigma$ there exists $W'\in\mathcal V_{(G,\sigma')}(e_G)$ such that $W'\subseteq V$. Being $[g,-]:(G,\sigma')\to(G,\sigma')$ continuous, there exists $W''\in\mathcal V_{(G,\sigma')}(e_G)$ such that $[g,W'']\subseteq W$. Let $W_g=W'\cap W''$. Then $[g,W_g]\subseteq U\cap W$. This completes the proof that $1_G:(G,\tau \vee\sigma')\to (G,\sigma')$ is semitopological.
\end{proof}

\begin{theorem}\label{pullback}
The class $\mathcal S_i$ is closed under pullback with respect to continuous injective homomorphisms.
\end{theorem}
\begin{proof}
Let $G$ be a group and $N$ a subgroup of $G$. Let $\tau\geq\sigma$ be group topologies on $G$ and $\sigma'$ a group topology on $N$ such that $\sigma\restriction_N\leq \sigma'$. If $1_G:(G,\tau)\to(G,\sigma)$ is semitopological, then $1_N:(N,\tau\restriction_N\vee\sigma')\to(N,\sigma')$ is semitopological as well. In fact by Theorem \ref{subgroups_stability} (also by Theorem \ref{subgroupst}(a)) $1_N:(N,\tau\restriction_N)\to(N,\sigma\restriction_N)$ is semitopological. Then apply Lemma \ref{pullback_lemma}.
\end{proof}

\subsection{Stability under taking compositions}

Now we consider the behavior of semitopological homomorphisms with respect to compositions. Obviously $\mathcal S$ is not stable under taking compositions, since its subclass $\mathcal S_i$ is already not stable under taking compositions.

\begin{proposition}\label{compositions..}
Let $(G,\tau),(K,\rho),(H,\sigma)$ be topological groups, $f_1:(G,\tau)\to(K,\rho)$ an open continuous homomorphism and $f_2:(K,\rho)\to(H,\sigma)$ a continuous isomorphism. Then the following are equivalent:
\begin{itemize}
\item[(a)]$f=f_2\circ f_1$ is A-open;
\item[(b)]$f_2$ is semitopological.
\end{itemize}
\end{proposition}
\begin{proof}
Note that $f_1(f^{-1}(X))=f_2^{-1}(X)$, whenever $X\subseteq H$.

\smallskip
(a)$\Rightarrow$(b) Let $W\in\mathcal V_{(K,\rho)}(e_K)$. There exists $U\in\mathcal V_{(G,\tau)}(e_G)$ such that $f_1(U)=W$. There exists $V\in\mathcal V_{(H,\sigma)}(e_H)$ such that $f^{-1}(V)$ is $f^{-1}(f(U))$-thin. This means that there exists $U'\in\mathcal V_{(G,\tau)}(e_G)$ such that $x U' x^{-1}\subseteq f\amu(f(U))$ for every $x\in f^{-1}(V)$. Then, since $f_2$ is an isomorphism, $$f_1(x)f_1(U')f_1(x)^{-1}\subseteq f_1(f^{-1}(f(U)))=f_1(U)=W$$ for every $x\in f^{-1}(V)$, where $f_1(U')\in\mathcal V_{(K,\rho)}(e_K)$ because $f_1$ is open. Since $f_1(x)\in f_1(f^{-1}(V))=f_2^{-1}(V)$, we have $y f_1(U') y^{-1}\subseteq f_1(U)$ for every $y\in f_2^{-1}(V)$, that is $f_2^{-1}(V)$ is $W$-thin.

Let $k\in K$. There exists $g\in G$ such that $f_1(g)=k$. There exists $V_g\in\mathcal V_{(H,\sigma)}(e_H)$ such that $[g,f^{-1}(V_g)]\subseteq f^{-1}(f(U))$. Then, since $f_2$ is an isomorphism, $$[k,f_2\amu(f_2(V_g))]=[f_1(g),f_1(f^{-1}(V_g))]\subseteq f_1(f^{-1}(f(U)))=f_1(U)=W.$$ This completes the proof that $f_2$ is A-open.

\smallskip
(b)$\Rightarrow$(a) Let $U\in\mathcal V_{(G,\tau)}(e_G)$. Since $f_1$ is open, $f_1(U)\in\mathcal V_{(K,\rho)}(e_K)$. There exists $V\in\mathcal V_{(H,\sigma)}(e_H)$ such that $f_2^{-1}(V)$ is $f_1(U)$-thin. This means that there exists $U'\in\mathcal V_{(G,\tau)}(e_G)$ such that $y f_1(U') y^{-1}\subseteq f_1(U)$ for every $y\in f_2^{-1}(V)=f_1(f^{-1}(V))$. Then $y\in f_2^{-1}(V)$ if and only if there exists $x\in f^{-1}(V)$ such that $y=f_1(x)$ and $f_1(x)f_1(U')f_1(x)^{-1}\subseteq f_1(U)$ for every $x\in f^{-1}(V)$. Therefore, since $f_2$ is an isomorphism, $$x U' x^{-1}\subseteq f_1^{-1}(f_1(x)f_1(U')f_1(x)^{-1})\subseteq f_1^{-1}(f_1(U))=f^{-1}(f(U))$$ for every $x\in f^{-1}(V)$. Hence $f^{-1}(U)$ is $U$-thin.

Let $g\in G$. There exists $V_g\in\mathcal V_{(H,\sigma)}(e_H)$ such that $[f_1(g),f_2^{-1}(V_g)]\subseteq f_1(U)$. Hence, since $f_2$ is an isomorphism, $$[g,f^{-1}(V_g)]\subseteq f_1^{-1}([f_1(g),f_2^{-1}(V_g)])\subseteq f_1^{-1}(f_1(U))=f^{-1}(f(U))$$ because $f_1(f^{-1}(V_g))=f_2^{-1}(V_g)$. So $f$ is A-open.
\end{proof}

In general the composition of semitopological homomorphisms is not semitopological, but one can wonder if $f$ of the above proposition has to be semitopological in case $f_2$ is semitopological (see Question \ref{compositions....}(a)).

Moreover Proposition \ref{compositions..} has the following corollary, which characterizes when a continuous surjective group homomorphism is A-open:

\begin{corollary}\label{Aop=Aop=semitop}
Let $f:(G,\tau)\to(H,\sigma)$ be a continuous surjective homomorphism and $q:G\to G/\ker f$ the canonical projection. By the first homomorphism theorem for topological groups $f':(G/\ker f,\tau_q)\to (H,\sigma)$, defined by $f'(q(g))=f(g)$ for every $g\in G$, is a continuous isomorphism. Then the following are equivalent:
\begin{itemize}
\item[(a)]$f$ is A-open;
\item[(b)]$f'$ is semitopological.
\end{itemize}
\end{corollary}

As a corollary of Theorem \ref{initial->all_equivalent} and Corollary \ref{Aop=Aop=semitop} we have the next result:

\begin{corollary}\label{initial->all_equivalent_corollary}
Let $(G,\tau),(H,\sigma)$ be topological groups and $f:(G,\tau)\to(H,\sigma)$ a continuous surjective homomorphism such that $\ker f\subseteq N_\tau$. Moreover let $q:G\to G/\ker f$ be the canonical projection and $f':(G/\ker f,\tau_q)\to (H,\sigma)$, defined by $f'(q(g))=f(g)$ for every $g\in G$. The following are equivalent:
\begin{itemize}
    \item[(a)]$f$ is strongly A-open;
    \item[(b)]$f$ is strongly A$^*$-open;
    \item[(b)]$f$ is A$^*$-open;
    \item[(c)]$f$ is A-open;
    \item[(d)]$f$ is semitopological;
    \item[(e)]$f'$ semitopological.
\end{itemize}
\end{corollary}

Moreover we have the following consequence of Corollary \ref{Aop=Aop=semitop} and Corollary \ref{hausdorff}.

\begin{corollary}
Let $f:(G,\tau)\to(H,\sigma)$ be an A-open continuous surjective homomorphism of topological groups. If $\tau$ is Hausdorff then $N_\sigma$ is abelian.
\end{corollary}
\begin{proof}
Consider the quotient $(G/\ker f,\tau_q)$, the canonical projection $q:G\to G/\ker f$ and $$f':(G/\ker f)\to H\ \ \text{defined by}\ \ f'(q(g))=f(g)\ \ \text{for every}\ \ g\in G.$$ By the first homomorphism theorem for topological groups $f':(G/\ker f,\tau_q)\to (H,\sigma)$ is a continuous isomorphism and it is semitopological thanks to Corollary \ref{Aop=Aop=semitop}.
By Theorem \ref{subgroupst}(a) $$f'\restriction_{(f')^{-1}(N_\sigma)}:((f')^{-1}(N_\sigma),\tau\restriction_{(f')^{-1}(N_\sigma)})\to(N_\sigma,\sigma\restriction_{N_\sigma})$$ is semitopological. Note that $\sigma\restriction_{N_\sigma}=\iota_{N_\sigma}$ and apply Corollary \ref{hausdorff} to conclude that $N_\sigma$ is abelian.
\end{proof}

In Proposition \ref{compositions..} we have considered $f=f_2\circ f_1$ where $f_1$ is an open continuous surjective homomorphism and $f_2$ a continuous isomorphism. Now we consider the opposite case, that is when $f_1$ is a continuous isomorphism and $f_2$ an open homomorphism, and see that the situation is not symmetric. In fact the following example shows that one of the implications does not hold true.

\begin{example}
Let $G$ be a topological group which is not abelian and $H=G/G'$. Moreover let $f_1=1_G:(G,\delta_G)\to(G,\iota_G)$ and $f_2=q:(G,\iota_G)\to(H,\iota_H)$ the canonical projection. By Corollary \ref{graph} $f=f_2\circ f_1$ is semitopological because $H$ is abelian, while $f_1$ is not semitopological by Corollary \ref{diab}.
\end{example}

The converse implication holds. Together with Proposition \ref{compositions..} we obtain the following result.

\begin{proposition}\label{compositions...}
Let $(G,\tau),(K,\rho),(H,\sigma)$ be topological groups, $f_1:(G,\tau)\to(K,\rho)$ and $f_2:(K,\rho)\to(H,\sigma)$ continuous surjective homomorphisms. Then $f=f_2\circ f_1$ is A-open whenever either $f_1$ is open and $f_2\in\mathcal S_i$ or $f_1\in\mathcal S_i$ and $f_2$ is open.
\end{proposition}
\begin{proof}
If $f_1$ is open and $f_2\in\mathcal S_i$ apply Proposition \ref{compositions..}. Suppose that $f_1\in\mathcal S_i$ and $f_2$ is open.
Being $f_1$ an isomorphism, as noted in the introduction we can suppose without loss of generality that $G=K$, $f_1=1_G$ and that $\tau\geq\rho$ are group topologies on $G$. Let $U\in\VG$. There exists $W\in\mathcal V_{(G,\rho)}(e_G)$ such that $W$ is $U$-thin, i.e. there exists $U'\in\VG$ such that $x U' x\amu \subseteq U$ for all $x\in W$. Let $f(W)=f_2(W)=V\in\VH$. Let us prove that $f\amu(V)$ is $f\amu(f(U))$-thin. Let $y\in f\amu(V)=f\amu(f(W))$. There exists $x\in W$ such that $f(x)=f(y)$ Then $$f(y U' y\amu)=f(y) f(U') f(y)\amu=f(x)f(U')f(x)\amu=f(x U' x\amu)\subseteq U.$$ Hence $y U y\amu\subseteq f\amu (f(U))$.

Let $g\in G$. There exists $W_g\in\mathcal V_{(G,\rho)}(e_G)$ such that $[g,W_g]\subseteq U$. Let $f(W_g)=f_2(W_g)=V_g\in\VH$. Since $$f([g,f\amu(V_g)])=[f(g),f(W_g)]=f([g,W_g])\subseteq f(U),$$ then $[g,f\amu(V_g)]\subseteq f\amu(f(U))$.
\end{proof}

As noted also after Proposition \ref{compositions..} one can ask if the conclusion of this proposition could be that $f$ is semitopological (see Question \ref{compositions....}(b)).

\medskip
We can prove the following result about compositions of strongly A-open continuous surjective homomorphisms, which shows the left cancelability of the class of all strongly A-open continuous surjective group homomorphisms. In view of Remark \ref{iso->all_equivalent} this yields the left cancelability of the class $\mathcal S_i$.

\begin{proposition}\label{strAop-left_canc}
Let $f_1:(G,\tau)\to(K,\rho)$ and $f_2:(K,\rho)\to(H,\sigma)$ be continuous surjective group homomorphisms and $f=f_2\circ f_1$.
\begin{itemize}
	\item[(a)]If $f$ is strongly A-open, then $f_1$ is strongly A-open.
	\item[(b)]If $f_1$ is open and $f$ is strongly A-open, then $f_2$ is strongly A-open.
\end{itemize}
\end{proposition}
\begin{proof}
(a) Suppose that $f$ is strongly A-open. Let $U\in\VG$. There exists $V\in\VH$ such that $f\amu(V)$ is $U$-thin, i.e. there exists $U'\in\VG$ such that $x U' x\amu\subseteq U$ for all $x\in f\amu(V)$. Let $W=f_2\amu(V)\in\mathcal V_{(K,\rho)}(e_K)$. Then $f_1\amu(W)$ is $U$-thin. In fact $f_1\amu(W)=f\amu(V)$ and so $x U' x\amu\subseteq U$ for all $x\in f_1\amu(W)$.

Let $g\in G$. There exists $V_g\in\VH$ such that $[g,f\amu(V_g)]\subseteq U$. Let $W_g=f_2\amu(V_g)\in\mathcal V_{(K,\rho)}(e_K)$. Since $f_1\amu(W_g)=f\amu(V_g)$, so $[g,f_1\amu(W_g)]\subseteq U$. This completes the proof that $f_1$ is strongly A-open.

\smallskip
(b) Assume that $f_1$ is open and $f$ is strongly A-open. Let $W\in\mathcal V_{(K,\rho)}(e_K)$. There exists $U\in\VG$ such that $f_1(U)= W$. There exists $V\in\VH$ such that $f\amu(V)$ is $U$-thin, i.e. there exists $U'\in\VG$ such that $x U' x\amu\subseteq U$ for every $x\in f\amu(V)$. Then $f_2\amu(V)$ is $W$-thin. Indeed for $y\in f_2\amu(V)$ there is $x\in f\amu(V)=f_1\amu(f_2\amu(V))$ such that $f_1(x)=y$ and so $$y f_1(U') y\amu= f_1(x U' x\amu)\subseteq f_1(U)=W,$$ where $f_1(U')\in\mathcal V_{(K,\rho)}(e_K)$ since $f_1$ is open.

Let $k\in K$. There exists $g\in G$ such that $f_1(g)=k$. There exists $V_g\in\VH$ such that $[g,f\amu(V_g)]\subseteq U$. Then $$[k,f_2\amu(V_g)]=f_1([g,f\amu(V_g)])\subseteq f_1(U)=W.$$
This completes the proof that $f_2$ is strongly A-open.
\end{proof}

\begin{theorem}\label{left-canc}
Let $f_1:(G,\tau)\to(K,\rho)$ and $f_2:(K,\rho)\to(H,\sigma)$ be continuous group isomorphisms and $f=f_2\circ f_1$. If $f$ is semitopological, then $f_1$ is semitopological.
\end{theorem}

\begin{remark}
In the previous theorem we have proved the left cancelability of the class $\mathcal S_i$. What about the right cancelability? It holds in some particular case, but we conjecture that it could not hold in general in $\mathcal S_i$.

Let $f_1:(G,\tau)\to(K,\rho)$ and $f_2:(K,\rho)\to(H,\sigma)$ be continuous group isomorphisms and suppose that $f=f_2\circ f_1$ is semitopological. Then for every $W\in\mathcal V_{(K,\rho)}(e_K)$ and for every $k\in K$ there exists $V_k\in\VH$ such that $[k,V_k]\subseteq W$. But in the general case it turns out to be not so simple to verify the other condition for $f_2$ to be semitopological. By Lemma \ref{SINsemitop} this condition is automatically verified in case $(K,\rho)$ is SIN. Analogously it is verified in case $\sigma=\iota_H$ by Corollary \ref{indiso}. It could be more simple eventually to look for a counterexample in case $\tau=\delta_G$.
\end{remark}

\subsection{Stability under taking quotients and products}

The next proposition implies Theorem \ref{quotientst}(b) which states that $\mathcal S_i$ is closed under taking quotients.

\begin{proposition}\label{stabquozhom}
Let $(G,\tau),(H,\sigma)$ be topological groups, $N$ a normal subgroup of $H$ and $q:G\to G/f^{-1}(N)$, $q':H\to H/N$ the canonical projections. If $f:(G,\tau)\to(H,\sigma)$ is an A-open continuous surjective homomorphism, then $f':(G/f^{-1}(N),\tau_q)\to(H/N,\sigma_q)$, defined by $f'(q(g))=q'(f(g))$ for every $g\in G$, is a semitopological isomorphism.
\end{proposition}
\begin{proof}
Note that $q(f\amu(S))=(f')\amu(q'(S))$ whenever $S$ is a subset of $H$. Indeed let $x\in q(f\amu(S))$. Then there exists $w\in f\amu(S)$ such that $q(w)=x$ and $f(w)=s\in S$. So $f'(x)=f'(q(w))=q'(f(w))=q'(s)$ and hence $x\in (f')\amu(q'(S))$. Now let $x\in (f')\amu(q'(S))$. Then there exists $s\in S$ such that $f'(x)=q'(s)$ and there exists $w\in f\amu(S)$ such that $f(w)=s$ and $f'(x)=q'(s)=q'(f(w))=f'(q(w))$. Consequently $x=q(w)\in q(f\amu(S))$.

Let $U\in\VG$. There exists $V\in \VH$ such that $f\amu(V)$ is $f\amu(f(U))$-thin. This means that there exists $U'\in\VG$ such that $x U' x\amu\subseteq f\amu(f(U))$ for every $x\in f\amu(V)$. Therefore $q(x)q(U')q(x)\amu\subseteq q(U)$ for every $q(x)\in q(f\amu(V))=(f')\amu(q'(V))$. Then $(f')\amu(q'(V))$ is $q(U)$-thin. Let $g\in G$. There exists $V_g\in\VH$ such that $[g,f\amu(V_g)]\subseteq f\amu(f(U))$. Hence $q([g,f\amu(V_g)])=[q(g),q(f\amu(V_g))]=[q(g),(f')\amu(q'(V_g))]\subseteq q(U)$.
\end{proof}

Note that in this proposition we have supposed that $f$ is A-open, which is the weakest of the properties we consider, while the conclusion being for isomorphisms has the strongest property.

\medskip
As a consequence of the results about composition and Proposition \ref{stabquozhom} we can prove the following more general theorem. It shows that the class of all A-open continuous surjective homomorphisms is stable under pushout with respect to open continuous surjective homomorphisms.

\begin{theorem}\label{pushout}
Let $(G,\tau),(H,\sigma)$ be topological groups, $L$ a normal subgroup of $G$ and $q:G\to G/L$, $q':H\to H/f(L)$ the canonical projections. If $f:(G,\tau)\to(H,\sigma)$ is an A-open continuous surjective homomorphism, then $f'':(G/L,\tau_q)\to(H/f(L),\sigma_q)$, defined by $f''(q(g))=q'(f(g))$ for every $g\in G$ is A-open.
\end{theorem}
\begin{proof}
Consider the following commutative diagram, where $\varphi$ is the canonical projection, $f'$ is defined by $f'(\varphi(q(g)))=q'(f(g))$ for every $g\in G$ and quotients are endowed with the quotient topologies:
\begin{equation*}
\xymatrix{
G \ar[rr]^{f} \ar[d]^{q} & & H \ar[d]^{q'}\\
G/L \ar[rr]^{f''} \ar[dr]^{\varphi} & & H/f(L)\\
& G/f\amu(f(L))\ar[ur]^{f'} &
}
\end{equation*}
By Proposition \ref{stabquozhom} $f'$ is a semitopological isomorphism. Then $f''$ is A-open by Corollary \ref{Aop=Aop=semitop}.
\end{proof}

The next theorem generalizes Theorem \ref{productst} about the stability of the class $\mathcal S_i$ under taking products.

\begin{theorem}\label{productst_hom}
Let $\{(G_i,\tau_i):i\in I\}$ and $\{(H_i,\sigma_i):i\in I\}$ be families of topological groups and $f_i:(G_i,\tau_i)\to(H_i,\sigma_i)$ a continuous surjective homomorphism for every $i\in I$. Let $G=\prod_{i\in I}G_i$, $\tau=\prod_{i\in I}\tau_i$, $H=\prod_{i\in I}H_i$, $\sigma=\prod_{i\in I}\sigma_i$ and $f=\prod_{i\in I}f_i$.
\begin{itemize}
\item[(a)]If $f_i:(G_i,\tau_i)\to(H_i,\sigma_i)$ is semitopological for every $i\in I$, then $f:(G,\tau)\to(H,\sigma)$ is semitopological.
\item[(b)]$f_i:(G_i,\tau_i)\to(H_i,\sigma_i)$ is A-open for every $i\in I$ if and only if $f:(G,\tau)\to(H,\sigma)$ is A-open.
\item[(c)]$f_i:(G_i,\tau_i)\to(H_i,\sigma_i)$ is A$^*$-open for every $i\in I$ if and only if $f:(G,\tau)\to(H,\sigma)$ is A$^*$-open.
\item[(d)]$f_i:(G_i,\tau_i)\to(H_i,\sigma_i)$ is strongly A$^*$-open for every $i\in I$ if and only if $f:(G,\tau)\to(H,\sigma)$ is strongly A$^*$-open.
\item[(e)]$f_i:(G_i,\tau_i)\to(H_i,\sigma_i)$ is strongly A-open for every $i\in I$ if and only if $f:(G,\tau)\to(H,\sigma)$ is strongly A-open.
\end{itemize}
\end{theorem}
\begin{proof}
(a) For every $i\in I$ let $(\widetilde G_i,\widetilde\tau_i)$, with $\widetilde f_i:(\widetilde G_i,\widetilde\tau_i)\to(H_i,\sigma_i)$, be an A-extension of $f_i:(G_i\tau_i)\to(H_i,\sigma_i)$. Then $(\widetilde G,\widetilde\tau)=(\prod_{i\in I}\widetilde G_i,\prod_{i\in I}\widetilde\tau_i)$ with $\widetilde f=\prod_{i\in I}\widetilde f_i:(\widetilde G,\widetilde\tau)\to(H,\sigma)$ is an A-extension of $f:(G,\tau)\to(H,\sigma)$. In fact $G$ is a normal subgroup of $\widetilde G$, $\widetilde f$ is open and continuous and $\widetilde f\restriction_G=f$.

\smallskip
(b) Suppose that $f_i:(G_i,\tau_i)\to(H_i,\sigma_i)$ is A-open for every $i\in I$. Let $U=U_{i_1}\times\dots\times U_{i_n}\times\prod_{i\in I\setminus\{i_1,\dots,i_n\}}G_i\in\VG$. For every $j\in\{i_1,\dots,i_n\}$ there exists $V_j\in\mathcal V_{(H_j,\sigma_j)}(e_{H_j})$ such that $f_j\amu(V_j)$ is $f_j\amu(f_j(U_j))$-thin. Then $V=V_{i_1}\times\dots\times V_{i_n}\times\prod_{i\in I\setminus\{i_1,\dots,i_n\}}H_i\in\VH$ is such that $f\amu(V)$ is $f\amu(f(U))$-thin. Let $g=(g_i)_{i\in I}\in G$. For every $j\in\{i_1,\dots,i_n\}$ there exists $V_{g_j}\in\mathcal V_{(H_j,\sigma_j)}(e_{H_j})$ such that $[g_j,f_j\amu(V_{g_j})]\subseteq f_j\amu(f_j(U_j))$. Then $V_g=V_{g_{i_1}}\times\dots\times V_{g_{i_n}}\times\prod_{i\in I\setminus\{i_1,\dots,i_n\}}H_i\in\VH$ is such that $[g,f\amu(V_g)]\subseteq f\amu(f(U))$. So $f:(G,\tau)\to(H,\sigma)$ is A-open.

Assume that $f:(G,\tau)\to(H,\sigma)$ is A-open.
Then $f_i$ is A-open for every $i\in I$ in view of Theorem \ref{pushout}.

\smallskip
(c) Suppose that $f_i:(G_i,\tau_i)\to(H_i,\sigma_i)$ is A$^*$-open for every $i\in I$. Let $U=U_{i_1}\times\dots\times U_{i_n}\times\prod_{i\in I\setminus\{i_1,\dots,i_n\}}G_i\in\VG$ and for every $i\in I$ let $s_i:H_i\to G_i$ be a section of $f_i$; then $s=\prod_{i\in I}s_i$ is a section of $f$. For every $j\in\{i_1,\dots,i_n\}$ there exists $V_j\in\mathcal V_{(H_j,\sigma_j)}(e_{H_j})$ such that $s_j(V_j)$ is $U_j$-thin. Then $V=V_{i_1}\times\dots\times V_{i_n}\times\prod_{i\in I\setminus\{i_1,\dots,i_n\}}H_i\in\VH$ is such that $s(V)$ is $U$-thin. Let $g=(g_i)_{i\in I}\in G$. For every $j\in\{i_1,\dots,i_n\}$ there exists $V_{g_j}\in\mathcal V_{(H_j,\sigma_j)}(e_{H_j})$ such that $[g_j,s_j(V_{g_j})]\subseteq U_j$. Then $V_g=V_{g_{i_1}}\times\dots\times V_{g_{i_n}}\times\prod_{i\in I\setminus\{i_1,\dots,i_n\}}H_i\in\VH$ is such that $[g,s(V_g)]\subseteq U$. Then $f:(G,\tau)\to(H,\sigma)$ is A$^*$-open.

Assume that $f:(G,\tau)\to(H,\sigma)$ is A$^*$-open and let $s:H\to G$ be the section of $f$ that witnesses that $f$ is A$^*$-open. Let $i\in I$ and $U_i\in\mathcal V_{(G_i,\tau_i)}(e_{G_i})$. Note that $s_i=s\restriction_{H_i}$ is a section of $f_i$. Then $U=U_i\times\prod_{j\in I\setminus\{i\}}G_j\in \VG$. There exists $V=\prod_{j\in I}V_j\in\VH$, where $V_j\in\mathcal V_{(H_j,\sigma_j)}(e_{H_j})$ and $V_j=H_j$ for all but finitely many $j\in I$, such that $s(V)$ is $U$-thin. This means that there exists $U'=\prod_{i\in I}U'_i\in \VG$ such that $s(x)U's(x)\amu\subseteq U$ for all $x\in V$. In particular $s(x_i)U'_i s(x_i)\amu\subseteq U\cap G_i=U_i$ for all $x_i\in V_i$ and so $s_i(V_i)$ is $U_i$-thin. Let $g_i\in G_i$ and $g=(g_j)_{j\in I}$ such that $g_j=e_{G_j}$ for each $j\neq i$. There exists $V_{g}=\prod_{j\in I}(V_g)_i\in\VH$, where $(V_g)_j\in\mathcal V_{(H_j,\sigma_j)}(e_{H_j})$ and $(V_g)_j=H_j$ for all but finitely many $j\in I$, such that $[g,s(V_g)]\subseteq U$. Since $s_i((V_g)_i)\subseteq s(V_g)$ and $[g_i,s_i((V_g)_i)]\subseteq G_i$, therefore $[g_i,s_i((V_g)_i)]\subseteq U_i$. This proves that $f_i$ is A$^*$-open.

\smallskip
(d) Suppose that $f_i:(G_i,\tau_i)\to(H_i,\sigma_i)$ is strongly A$^*$-open for every $i\in I$. Then for every $i\in I$ there exists a section $s_i:H_i\to G_i$ that witnesses that $f_i$ is strongly A$^*$-open. As proved in (c), $s=\prod_{i\in I}s_i$ is a section of $f$ that witnesses that $f$ is A$^*$-open. It remains to prove that $s$ is a quasihomomorphism. This is true since $s'=\prod_{i\in I}s_i'$ and for $y=(y_i)\in H$ $s_y=\prod_{i\in I}s_{y_i}$.

Assume that $f:(G,\tau)\to(H,\sigma)$ is strongly A$^*$-open. There exists a section $s$ of $f$ that witnesses that $f$ is strongly A$^*$-open. Let $i\in I$ and $s_i=s\restriction H_i:H_i\to G_i$. As proved in (c) $s_i$ is a section of $f_i$ that witnesses that $f_i$ is A$^*$-open. It remains to prove that $s_i:(H_i,\sigma_i)\to(G_i,\tau_i)$ is a quasihomomorphism. Let $U_{i_0}\in\mathcal V_{(G_{i_0},\tau_{i_0})}(e_{G_{i_0}})$. Then $U=U_{i_0}\times\prod_{i\in I\setminus \{i_0\}}\in\VG$. There exists $V=\prod_{i\in I}V_i\in\VH$ such that $s'(V,V)\subseteq U$. In particular $s_{i_0}'(V_{i_0})\subseteq U_{i_0}$. So $s'_{i_0}$ is continuous at $(e_{H_{i_0}},e_{H_{i_0}})$. In the same way it can be proved that $(s_i)_y$ is continuous at $e_{H_{i_0}}$.

\smallskip
(e) Suppose that $f_i:(G_i,\tau_i)\to(H_i,\sigma_i)$ is strongly A-open for every $i\in I$. Let $U=U_{i_1}\times\dots\times U_{i_n}\times\prod_{i\in I\setminus\{i_1,\dots,i_n\}}G_i\in\VG$. For every $j\in\{i_1,\dots,i_n\}$ there exists $V_j\in\mathcal V_{(H_j,\sigma_j)}(e_{H_j})$ such that $f_j\amu(V_j)$ is $U_j$-thin. Then $V=V_{i_1}\times\dots\times V_{i_n}\times\prod_{i\in I\setminus\{i_1,\dots,i_n\}}H_i\in\VH$ is such that $f\amu(V)$ is $U$-thin. Let $g=(g_i)_{i\in I}\in G$. For every $j\in\{i_1,\dots,i_n\}$ there exists $V_{g_j}\in\mathcal V_{(H_j,\sigma_j)}(e_{H_j})$ such that $[g_j,f_j\amu(V_{g_j})]\subseteq U_j$. Then $V_g=V_{g_{i_1}}\times\dots\times V_{g_{i_n}}\times\prod_{i\in I\setminus\{i_1,\dots,i_n\}}G_i\in\VH$ is such that $[g,f\amu(V_g)]\subseteq U$.

Assume that $f:(G,\tau)\to(H,\sigma)$ is strongly A-open.
Then $f_i$ is strongly A-open for every $i\in I$ by Theorem \ref{subgroups_stability}(d), since $f_i=f\restriction_{G_i}$ for every $i\in I$.
\end{proof}

In (a) of this theorem it is not clear if the converse implication holds in general. In view of (b) and since semitopological coincides with A-open for isomorphisms, it holds true for isomorphisms (so also the converse implication of Theorem \ref{productst} holds). Moreover it holds true in the particular case of the following example.

\begin{example}
Let $G_1,G_2,H_1,H_2$ be topological groups and $f_i:G_i\to H_i$ a continuous surjective homomorphism for $i=1,2$. Moreover $G=G_1\times G_2$, $H=H_1\times H_2$ and $f=f_1\times f_2$. If $f$ is semitopological and $G_2$ is a normal subgroup of the A-extension $\widetilde G$, then $f_1$ is semitopological. In fact in this case $G_1$ is topologically isomorphic to a normal subgroup of $\widetilde G/G_2$. Moreover let $\widetilde f_1:\widetilde G/G_2\to H_1$ be the homomorphism induced by the open continuous homomorphism $\widetilde f$ associated to $\widetilde G$. Hence $\widetilde f_1$ is open.

If $Z(G_1)$ is trivial and $G_2$ is abelian, then $G_2=Z(G)$ and in particular $G_2$ is normal in $\widetilde G$.
\end{example}

\section{Open questions}\label{questions}

We sum up in the following graph the relations that we know and do not know among the properties we have introduced and the already defined ones (for continuous surjective homomorphisms of topological groups). We give here more properties that those in Diagram \ref{small_diagram} on page \pageref{small_diagram} to give a complete description. The examples used to prove that some implications do not hold are Examples  \ref{open,A*-open,non-strongly-A-open}, \ref{non-open,str_A-open}, \ref{strAop-nonquasihomo}, \ref{non-open_*-section} and \ref{Aop-<A*op}.

\begin{equation*}
\xymatrix{
\text{semitopological+strongly A-open} \ar[dd]^{no} & \text{strongly A$^*$-open} \ar@{=>}[lddd] \ar@{=>}[d] \ar@{=>}[rddd] \ar[rdd]^{no} \ar[ldd]_{no} & \text{open+(strongly) A$^*$-open} \ar[dd]^{no} \\
& \text{$\exists$ section quasih.} \ar@/^0.7pc/[ddd]^{no} & \\
\text{open} \ar@{=>}[d] \ar[rr]^{no} \ar@/^0.3pc/@{-->}[ru]^{?}  \ar@/_0.8pc/@{-->}[rrd]^? &  & \text{strongly A-open} \ar@{=>}[d] \ar@/_0.3pc/[lu]_{no} \ar@/_1pc/[ll]_{no} \ar@/_0.7pc/@{-->}[lld]_? \\
\text{semitopological} \ar@/^0.1pc/[uur]^{no} \ar@{=>}[dr] \ar@/^0.7pc/[u]^{no} \ar@{-->}@/_0.7pc/[rr]_? \ar[rru]^{no} & & \text{A$^*$-open} \ar@{=>}[dl] \ar@/_0.7pc/[u]_{no} \ar@{-->}[ll]_? \\
& \text{A-open} \ar@{-->}@/^1pc/[lu]^? \ar@/_1pc/[ru]_{no} &
}
\end{equation*}

We collect here all questions left open inside this paper.
Note that only one of the distinct conditions A-open, A$^*$-open and strongly A-open could eventually imply semitopological. And only A-open and A$^*$-open could eventually be equivalent to semitopological.

\begin{question}
Does either A-open or A$^*$-open or strongly A-open imply semitopological for a continuous surjective homomorphism of topological groups? Is either A-open or A$^*$-open equivalent to semitopological? Is this true in case the domain is discrete or the codomain indiscrete?
\end{question}

A positive answer to the following questions would be a first step in understanding whether open yields (strongly) A$^*$-open for continuous surjective homomorphisms of topological groups.

\begin{question}
Does open imply the existence of a section which is a quasihomomorphism?
\end{question}

Proposition \ref{open+s-hom->strA*op} suggests to consider the following question.

\begin{question}
Is every open continuous surjective homomorphism A$^*$-open?
\end{question}

The next question arises from Propositions \ref{compositions..} and \ref{compositions...}.

\begin{question}\label{compositions....}
Let $f_1:(G,\tau)\to(K,\rho)$ and $f_2:(K,\rho)\to(H,\sigma)$ be continuous surjective homomorphisms.
\begin{itemize}
	\item[(a)]If $f_1$ is open and $f_2$ is a semitopological isomorphism, must $f=f_2\circ f_1$ be semitopological?
	\item[(b)]If $f_1$ is a semitopological isomorphism and $f_2$ is open, must $f=f_2\circ f_1$ be semitopological?
\end{itemize}
\end{question}

Moreover we remind here the open problems left in \cite{Ar} about semitopological isomorphisms, that we discuss in \cite{DG2}.

\begin{problem}\emph{\cite[Problem 13]{Ar}}
Let $\mathcal K$ be a class of topological groups. Find $(G,\tau)\in\mathcal K$ such that every semitopological isomorphism $f:(G,\tau)\to(H,\sigma)$, where $(H,\sigma) \in \mathcal K$, is open.
\end{problem}

\begin{problem}\emph{\cite[Problem 14]{Ar}}
Find groups $G$ such that for all group topologies $\tau$ on $G$ every semitopological isomorphism $f:(G,\tau)\to(H,\sigma)$, where $(H,\sigma)$ is another topological group, is open.
\end{problem}

\begin{problem}\emph{\cite[Problem 15]{Ar}}
\begin{itemize}
	\item Which are the continuous isomorphisms of topological groups that are compositions of semitopological isomorphisms?
	\item Is every continuous isomorphism of topological groups composition of semitopological isomorphisms?
\end{itemize}
\end{problem}

The same problems could be posed for semitopological and d-semitopological homomorphisms.

\subsubsection*{Acknowledgments}
I am grateful to Professor Dikran Dikranjan for his helpful comments and suggestions. Moreover I want to thank the referee for his/her suggestions.

\end{document}